\documentclass[12pt,reqno]{amsart}

\newcommand\version{December 17, 2024}

\usepackage{amsmath,amsfonts,amsthm,amssymb,amsxtra}
\usepackage{bbm} 
\usepackage{hyperref}	

\usepackage{xfrac}
\usepackage{todonotes, relsize, pgfplots, tikz, float}
\usetikzlibrary{calc,decorations.markings}
\tikzset{
	> = stealth,
	every pin/.style = {pin edge = {}},
	flow/.style = {decoration = {markings, mark=at position #1 with {\arrow{>}}},
		postaction = {decorate}
	},
	flow/.default = 0.5,
	main/.style = {color=#1, line width=0.5pt, line cap=round, line join=round},
	main/.default = black,
	fontscale/.style={font=\relsize{#1}},
}

\newtheorem{theorem}{Theorem}
\newtheorem{proposition}[theorem]{Proposition}
\newtheorem{lemma}[theorem]{Lemma}

\theoremstyle{definition}

\theoremstyle{remark}

\newtheorem{remarks}[theorem]{Remarks}

\numberwithin{equation}{section}

\renewcommand{\epsilon}{\varepsilon}

\newcommand{\loc}{{\rm loc}}

\newcommand{\N}{\mathbb{N}}

\renewcommand{\phi}{\varphi}
\newcommand{\R}{\mathbb{R}}

\newcommand{\Sph}{\mathbb{S}}

\newcommand{\lo}{\mathrm{lo}}
\newcommand{\me}{\mathrm{med}}
\newcommand{\hi}{\mathrm{hi}}

\usepackage{mathtools, todonotes, thmtools}\usepackage{tikz}
\usetikzlibrary{datavisualization}
\usetikzlibrary{datavisualization.formats.functions}
\usetikzlibrary{patterns}

\let\oldtocsection=\tocsection

\let\oldtocsubsection=\tocsubsection

\let\oldtocsubsubsection=\tocsubsubsection

\renewcommand{\tocsection}[2]{\hspace{0em}\oldtocsection{#1}{#2}}
\renewcommand{\tocsubsection}[2]{\hspace{1em}\oldtocsubsection{#1}{#2}}
\renewcommand{\tocsubsubsection}[2]{\hspace{2em}\oldtocsubsubsection{#1}{#2}}

\usepackage{todonotes}
\setuptodonotes{color=white, bordercolor=white, textcolor=blue, size=\normalsize}
\newif\ifTodo
\Todotrue

\begin{document}

	\begin{titlepage}
		\huge \title[Quantitative $\sigma_2$-curvature inequality]{The sharp $\sigma_2$-curvature inequality on the sphere\\ in quantitative form}
		\vspace{7cm}
	\end{titlepage}
		
	\author{Rupert L. Frank}
	\address[Rupert L. Frank]{Mathe\-matisches Institut, Ludwig-Maximilians-Universit\"at M\"unchen, The\-resienstr.~39, 80333 M\"unchen, Germany, and Munich Center for Quantum Science and Technology, Schel\-ling\-str.~4, 80799 M\"unchen, Germany, and Mathematics 253-37, Caltech, Pasa\-de\-na, CA 91125, USA}
	\email{r.frank@lmu.de}
	
	\author{Jonas W.~Peteranderl}
	
	\address[Jonas W.~ Peteranderl]{Mathematisches Institut, Ludwig-Maximilians-Universit\"at M\"unchen, Theresienstr.~39, 80333 M\"unchen, Germany}	\email{peterand@math.lmu.de}
	
	\begin{abstract}
		Among all metrics on $\mathbb S^d$ with $d>4$ that are conformal to the standard metric and have positive scalar curvature, the total $\sigma_2$-curvature, normalized by the volume, is uniquely (up to M\"obius transformations) minimized by the standard metric. We show that if a metric almost minimizes, then it is almost the standard metric (up to M\"obius transformations). This closeness is measured in terms of Sobolev norms of the conformal factor, and we obtain the optimal stability exponents for two different notions of closeness. This is a stability result for an optimization problem whose Euler--Lagrange equation is fully nonlinear.
\end{abstract}

\date{\version}
\thanks{\copyright\, 2024 by the authors. This paper may be reproduced, in its entirety, for non-commercial purposes.\\
Partial support through US National Science Foundation grant DMS-1954995 (R.L.F.), as well as through the German Research Foundation through EXC-2111-390814868 (R.L.F.) and TRR 352-Project-ID 470903074 (R.L.F.~\&~J.W.P.) is acknowledged.}
	
	\maketitle
	\setcounter{page}{1}
	
\section{Introduction and main result}
	
\subsection{The sharp $\sigma_2$-curvature inequality of Guan--Viaclovsky--Wang} 

We consider Riemannian metrics $g$ on the $d$-dimensional unit sphere $\Sph^d$, $d\geq 5$, that are conformally equivalent to the standard metric $g_*$ of constant sectional curvature one. We are interested in their $\sigma_1$- and $\sigma_2$-curvatures, which are defined, in terms of their Ricci curvature $\operatorname{Ric}^{g}$ and their scalar curvature $R^g$, by
$$
\sigma^g_1\coloneqq \frac{1}{2(d-1)} \, R^g
\qquad\text{and}\qquad
\sigma^g_2\coloneqq \frac{1}{2(d-2)^2}\left(\frac{d}{4(d-1)} (R^{g})^2-|\operatorname{Ric}^g|^2\right).
$$
These are particular cases of $\sigma_k$-curvatures, $k=1,\ldots,d$, which were introduced by Viaclovsky \cite{Viaclovsky2000} via elementary symmetric functions of the Schouten tensor.

Of fundamental importance in conformal geometry is the Yamabe (or Einstein--Hilbert) functional 
	$$
	\mathcal F_1[g] := \frac{1}{\operatorname{vol}(g)^{\frac{d-2}{d}}}\int_{\Sph^d} \sigma_1^{g} \, \mathrm d\operatorname{vol}_g \,.
	$$
	In the present paper we study the $\sigma_2$-analogue of the Yamabe functional, given by 
	$$
	\mathcal F_2[g] := \frac{1}{\operatorname{vol}(g)^{\frac{d-4}{d}}}\int_{\Sph^d} \sigma_2^{g} \, \mathrm d\operatorname{vol}_g \,.
	$$
	
An important feature of both functionals is their conformal invariance. That is, if $\Psi$ is a M\"obius transformation (that is, a conformal diffeomorphism of $(\Sph^d,g_*)$), then
$$
\mathcal F_1[\Psi^*g] = \mathcal F_1[g]
\qquad\text{and}\qquad
\mathcal F_2[\Psi^*g] = \mathcal F_2[g] \,.
$$

We are interested in minimizing the functionals $\mathcal F_1$ and $\mathcal F_2$ over certain classes of metrics on $\Sph^d$. A well-known result, due to Rodemich \cite{Rodemich1966}, Aubin \cite{Aubin1976}, and Talenti \cite{Talenti1976} (and, up to the existence of an optimizer, also Obata \cite{Obata1971}), states that among all metrics on $\Sph^d$ that are conformal to $g_*$ the functional $\mathcal F_1$ attains its minimum precisely at those metrics that are obtained from $g_*$ by a M\"obius transformation. In particular, we have the sharp inequality
	\begin{equation}
		\label{eq:sharp1intro}
		\frac{1}{\operatorname{vol}(g)^{\frac{d-2}{d}}}\int_{\Sph^d} \sigma_1^{g} \, \mathrm d\operatorname{vol}_g \geq S_d^{(1)}
	\end{equation}
	for any metric $g$ on $\Sph^d$ that is conformally equivalent to $g_*$, where
	$$
	S_d^{(1)} := \frac{1}{\operatorname{vol}(g_*)^{\frac{d-2}{d}}}\int_{\Sph^d} \sigma_1^{g_*} \, \mathrm d\operatorname{vol}_{g_*} = \frac d2 \, |\Sph^d|^\frac 2d
	$$
	with $\operatorname{vol}(g_*)=|\Sph^d|$ and $\sigma_1^{g_*}=d/2$. It is well known and often referred to as Yamabe inequality. As will be recalled below, \eqref{eq:sharp1intro} is in fact a restatement of the sharp Sobolev inequality.
	
	The corresponding result is valid for $\mathcal F_2$ as well and is due to Guan, Viaclovsky, and Wang \cite{Guan2003b}. Namely, among all metrics on $\Sph^d$ that are conformally equivalent to $g_*$ and have positive scalar curvature the functional $\mathcal F_2$ attains its minimum precisely at those metrics that are obtained from $g_*$ by a M\"obius transformation. In particular, we have the sharp inequality
	\begin{equation}
		\label{eq:sharpintro}
		\frac{1}{\operatorname{vol}(g)^{\frac{d-4}{d}}}\int_{\Sph^d} \sigma_2^{g} \, \mathrm d\operatorname{vol}_g \geq S_d^{(2)}
	\end{equation}
	for any metric $g$ on $\Sph^d$ that is conformally equivalent to $g_*$ and has positive scalar curvature, where
	$$
	S_d^{(2)} := \frac{1}{\operatorname{vol}(g_*)^{\frac{d-4}{d}}}\int_{\Sph^d} \sigma_2^{g_*} \, \mathrm d\operatorname{vol}_{g_*} = \frac{d(d-1)}{8} \, |\Sph^d|^\frac 4d
	$$ with $\sigma_2^{g_*}=d(d-1)/8$. Originally, this theorem was proved for metrics having, in addition, positive $\sigma_2$-curvature, but using a result of Ge and Wang \cite{Ge2013} (see inequality \eqref{eq:gewang} below) this additional constraint can be omitted. We also refer to the recent proof of the sharp $\sigma_2$-curvature inequality due to Case \cite{Case2020}, which extends a method of \cite{Frank2012a,Frank2012b} to the fully nonlinear setting; see also \cite{Case2020a}.

In the present paper we are interested in the \emph{stability} of the sharp inequality \eqref{eq:sharpintro}. That is, we consider positive scalar curvature metrics $g$ on $\Sph^d$ that are conformally equivalent to $g_*$ and investigate to which extent closeness of $\mathcal F_2[g]$ to $S_d^{(2)}$ implies closeness of $g$ to $\Psi^*g_*$ for some M\"obius transformation $\Psi$. Of course, we still need to specify the notion of closeness that we are using.
	
	The corresponding question for the $\mathcal F_1$-minimization problem has been asked, in an equivalent form, by Brezis and Lieb \cite{BREZIS198573} and was answered in a celebrated paper of Bianchi and Egnell \cite{Bianchi1991}.
	
	In order to formulate the stability result of Bianchi and Egnell, and ours as well, we phrase the optimization problems for $\mathcal F_1$ and $\mathcal F_2$ in an analytic manner. The metric $g$ that is conformally equivalent to $g_*$ can be written as a positive, smooth function times $g_*$. We choose two different parametrizations in the two different cases, namely,
	$$
	g = w^\frac{4}{d-2} g_*
	\qquad\text{and}\qquad
	g = u^\frac{8}{d-4} g_*
	$$
	in the $\mathcal F_1$- and $\mathcal F_2$-problems, respectively. Here $w$ and $u$ are smooth, positive functions on $\Sph^d$. Then
	$$
	\operatorname{vol}\left(w^\frac{4}{d-2} g_*\right) = \int_{\Sph^d} w^\frac{2d}{d-2}\,\mathrm d\omega
	\qquad\text{and}\qquad
	\operatorname{vol}\left(u^\frac{4}{d-4} g_*\right) = \int_{\Sph^d} u^\frac{4d}{d-4}\,\mathrm d\omega \,,
	$$
	where $\mathrm d\omega$ denotes the volume form on $\Sph^d$ with respect to the standard metric $g_*$. Moreover, it is well known that
	$$
	\int_{\Sph^d} \sigma_1^{w^\frac{4}{d-2} g_*} \, \mathrm d\operatorname{vol}_{w^\frac{4}{d-2} g_*} = \frac{2}{d-2} \int_{\Sph^d} \left( |\nabla w|^2 + \frac{d(d-2)}{4} w^2 \right) \mathrm d\omega \,,
	$$
	where $\nabla$ denotes covariant differentiation. Thus, 
	$$
	\mathcal F_1[w^\frac{4}{d-2} g_*] = \frac{2}{d-2} \frac{\int_{\Sph^d} \left( |\nabla w|^2 + \frac{d(d-2)}{4} w^2 \right) \mathrm d\omega}{\left( \int_{\Sph^d} w^\frac{2d}{d-2}\,\mathrm d\omega \right)^\frac{d-2}d} =: F_1[w] \,,
	$$
	and \eqref{eq:sharp1intro} is equivalent to the sharp Sobolev inequality
	$$
	F_1[w] \geq S_d^{(1)} \,.
	$$
	The conformal invariance of the $\mathcal F_1$-minimization problem means that, if for a function $w$ on $\Sph^d$ and a M\"obius transformation $\Psi$ with Jacobian denoted by $J_\Psi$ we set
	\begin{equation}\label{eq:conf2}
		[w]_\Psi\coloneqq J_\Psi^\frac{d-2}{2d}\ w\circ \Psi\,,
	\end{equation}
	then
	$$
	F_1[[w]_\Psi] = F_1[w] \,.
	$$
	The stability theorem of Bianchi and Egnell \cite{Bianchi1991} states that there is a constant $c_d>0$ such that
	\begin{equation}
		\label{eq:beintro}
		F_1[w] - S_d^{(1)} \geq c_d \inf_{ \lambda, \Psi}\|\lambda [w]_{\Psi}-1\|_{W^{1,2}(\mathbb S^{d})}^2 \,,
	\end{equation}
	where the infimum is taken over all $\lambda\in \R$ and all M\"obius transformations $\Psi$. Note that $\lambda [w]_{\Psi}=1$ for some $\lambda$ and $\Psi$ if and only if $F_1[w] = S_d^{(1)}$. In this sense, inequality \eqref{eq:beintro} gives a quantitative form of the sharp Sobolev inequality. Moreover, the right side shares the same conformal invariance as the left side. (A translation to $\Sph^d$ of the inequality in \cite{Bianchi1991} gives a slightly different but equivalent inequality. For our discussion here, the above form seems more natural.)
	
The right side of inequality \eqref{eq:beintro} is optimal with respect to the exponent $2$, in the sense that the inequality does not hold with an exponent that is smaller than $2$.
	
	\emph{Our main result in this paper is an analogue of inequality \eqref{eq:beintro} for the $\mathcal F_2$-minimization problem.} We first express the total $\sigma_2$-curvature of $g$ in terms of the function $u$ such that $g= u^\frac{8}{d-4}g_*$. We define
	\begin{equation*}
		\sigma_1(u)\coloneqq -\frac{d-4}{8}\Delta (u^2)-|\nabla u|^2+\frac{d}{2} \left(\frac{d-4}{4}\right)^2u^2
	\end{equation*}
	and note that, by a short computation,
	\begin{equation*}
		\sigma_1(u) = \left(\frac{d-4}{4}\right)^2u^{\frac{2d}{d-4}} \ \sigma_1^{u^\frac8{d-4} g_*}\,.
	\end{equation*}
	Next, we set
	\begin{equation}\label{eq:density}
		e_2 (u)\coloneqq \left(\frac{4}{d-4}\right)^3 \left(\sigma_1(u)+\frac{1}{2} |\nabla u|^2+\frac{d-2}{2}\left(\frac{d-4}{4}\right)^2u^2\right) |\nabla u|^2+\frac{d(d-1)}{8} u^4 \,.
	\end{equation}	
	It was observed by Case \cite{Case2020} that
	$$
	\int_{\Sph^d} \sigma_2^{u^\frac 8{d-4}g_*}\, \mathrm d \operatorname{vol}_{u^\frac 8{d-4}g_*} = \int_{\Sph^d} e_2(u)\,\mathrm d\omega \,.
	$$
	We emphasize that we do \emph{not} claim that $e_2(u)$ is equal to $u^{-\frac{4d}{d-4}} \ \sigma_2^{u^\frac 8{d-4}g_*}$. This identity holds only up to terms that vanish when integrated over $\Sph^d$. The advantage of the `energy density' $e_2(u)$ is that it is a sum of nonnegative terms. 
	
	This discussion shows that
	$$
	\mathcal F_2[u^\frac 8{d-4}g_*] = \frac{\int_{\Sph^d} e_2(u)\,\mathrm d\omega}{\left( \int_{\Sph^d} u^\frac{4d}{d-4}\,\mathrm d\omega \right)^\frac{d-4}d} =: F_2[u] \,,
	$$
	and that the sharp inequality \eqref{eq:sharpintro} is equivalent to the sharp Sobolev-type inequality
	\begin{equation}
		\label{eq:sharpintrofcn}
		F_2[u] \geq S_d^{(2)} \,,
	\end{equation}
	valid for all positive, smooth functions $u$ on $\Sph^d$ with $\sigma_1(u)> 0$.
	
	As a final preliminary, for a M\"obius transformation $\Psi$ of $\Sph^d$ we set	
	\begin{equation*}
		(u)_{\Psi}\coloneqq J_\Psi^{\frac{d-4}{4d}} \ u\circ \Psi\,.
	\end{equation*}
	This is similar but different from definition \eqref{eq:conf2}. Then the conformal invariance of the $\mathcal F_2$-minimization problem implies that
	$$
	F_2[(u)_\Psi] = F_2[u] \,.
	$$
	Moreover, the assertion that equality in \eqref{eq:sharpintro} holds precisely for metrics that are obtained from $g_*$ by M\"obius transformations is equivalent to the assertion that equality in \eqref{eq:sharpintrofcn} holds if and only if $\lambda (u)_\Psi=1$ for some $\lambda\in\R$ and some M\"obius transformation $\Psi$.
	
	
\subsection{Our main stability result}

We are now ready to formulate our main result, which states that if $g$ is a positive scalar curvature metric on $\Sph^d$ that is conformal to $g_*$ and has $\mathcal F_2[g]$ close to $\mathcal F_2[g_*]$, then $g$ or a M\"obius transformation thereof is close to $g_*$ in a quantitative sense. We phrase this in the functional reformulation \eqref{eq:sharpintrofcn} of \eqref{eq:sharpintro}.

\begin{theorem}[Quantitative stability]\label{thm}
	Let $d>4$. Then there is a constant $c_d>0$ such that for all $u\in C^\infty(\mathbb S^d)$ with $u>0$ and $\sigma_1(u)>0$ we have
	\begin{equation}
		\label{eq:quantstab}
		F_2[u] - S_d^{(2)} \geq c_d
		\inf_{\lambda,\Psi} \left( \|\lambda \, (u)_{\Psi}-1\|_{W^{1,2}(\mathbb S^{d})}^2 + \|\lambda \, (u)_{\Psi}-1\|_{W^{1,4}(\mathbb S^{d})}^4 \right),
	\end{equation}
	where the infimum is taken over all $\lambda\in \R$ and M\"obius transformations $\Psi:\mathbb S^d\to\mathbb S^d$.
\end{theorem}

Here $W^{1,p}=W^{1,p}(\mathbb S^d)$ denotes the Sobolev space with norm $\|\cdot\|_{W^{1,p}}\coloneqq (\|\nabla\cdot\|^p_p+\|\cdot \|_p^p)^{1/{p}}$.

\begin{remarks}
	We make a number of short remarks, which will be substantiated later.\\
	(a) As far as we know, this is the first stability result for a functional whose corresponding Euler--Lagrange equation is fully nonlinear. This should be contrasted with the stability for $\mathcal F_1$, whose Euler--Lagrange equation is semilinear, and the stability for the $p$-Sobolev inequality, whose Euler--Lagrange equation is quasilinear. The fully nonlinear nature is reflected in the appearance of the term involving $\sigma_1(u)$ in the energy density $e_2(u)$ in \eqref{eq:density}, which involves the Laplacian of $u$ and cannot be controlled by first derivatives. Also the appearance of the pointwise constraint $\sigma_1(u)>0$, involving again the Laplacian, makes \eqref{eq:quantstab} rather different from other existing stability inequalities.\\
	(b) On the right side of \eqref{eq:quantstab} two different norms appear, namely that in $W^{1,2}$ and that in $W^{1,4}$. Of course, $W^{1,4}$-stability is stronger than $W^{1,2}$-stability, but this stronger norm comes at the expense of a weaker power, namely, 4 instead of 2. Our result is optimal in the sense that $W^{1,4}$-stability does not hold with a smaller exponent than 4, and $W^{1,2}$-stability does not hold with a smaller exponent than 2. We emphasize that the two norms capture different notions of `small perturbations' of an optimizer. In particular, the scenarios showing optimality of the exponents 2 and 4 of the $W^{1,2}$ and $W^{1,4}$ are distinct. This will be further discussed in Section~\ref{sec:4}.\\
	(c) The appearance of the exponent 4 in the $W^{1,4}$ distance shares some similarities with the Sobolev inequality in the homogeneous Sobolev space $\dot W^{1,4}(\R^d)$ in quantitative form, due to Figalli and Zhang \cite{Figalli2022}. This is not a coincidence, given that the energy density $e_2(u)$ involves the term $|\nabla u|^4$. This resemblance becomes even clearer when the sharp $\sigma_2$-curvature inequality is written via stereographic projection on $\R^d$, where it becomes
	$$
	\left( \frac{4}{d-4}\right)^3 \int_{\R^d} \left( I(u) |\nabla u|^2 + \frac12 |\nabla u|^4 \right)\mathrm dx \geq S_d^{(2)} \left( \int_{\R^d} u^\frac{4d}{d-4}\,\mathrm dx \right)^\frac{d-4}{d}
	$$
	for $0<u\in C^\infty(\R^d)\cap \dot W^{1,4}(\R^d)$ such that $I(u)>0$ with
	$$
	I(u) := -\frac{d-4}8 \Delta (u^2)  - |\nabla u|^2 \,;
	$$ see \cite[Remark 2.5]{Case2020}. However, the latter inequality is conformally invariant, like the Sobolev inequality in $\dot W^{1,2}(\R^d)$ and \emph{unlike} the Sobolev inequality in $\dot W^{1,4}(\R^d)$. Our proof uses some ideas from \cite{Figalli2015,Figalli2022}, but at the same time we observe that the result in these papers can be strengthened; see Appendix \ref{app:figallizhang}. This answers a question that Robin Neumayer, to whom we are grateful, asked after a conference talk on the results in the present paper.\\
	(d) As we have noted before, the sharp $\sigma_2$-curvature inequality \eqref{eq:sharpintrofcn} is conformally invariant. The right side of \eqref{eq:quantstab} shares the same conformal invariance. This dictates our choice of the notion of distance to the family of optimizers. Among other variants that suggest themselves, the quantity $\inf_{\lambda,\Psi} \| u - \lambda (1)_\Psi\|_{W^{1,k}}$ with $k=2,4$ does not share the conformal invariance, and the quantity $\inf_{\lambda,\Psi} F_2[u-\lambda(1)_\Psi]$ is not necessarily nonnegative.\\
	(e) Our proof follows the Bianchi--Egnell strategy, that is, it consists in two steps, namely, in a global-to-local reduction and in a local analysis close to the set of optimizers. Both steps are somewhat different from other applications of this strategy. We discuss this further in Subsection \ref{sec:strategy}.
\end{remarks}


\subsection{Relation to other works}

Our analysis lies at the intersection of conformal geometry and the stability theory of functional inequalities. Let us provide some background and give some references.

\subsubsection*{Background from conformal geometry}	
Let $(M,g)$ be a Riemannian manifold of dimension $d\geq 3$ with Schouten tensor
\begin{equation*}
	A^g\coloneqq \frac{1}{d-2} \left(\operatorname{Ric}^g-\frac{1}{2(d-1)} R^g g\right).
\end{equation*}
The $\sigma_k^g$-curvature is defined as the $k$-th elementary symmetric polynomial of the eigenvalues of $A^g$ with respect to the metric $g$. In particular, for $k=1$ and $k=2$, this leads to the same formulas as at the beginning of the introduction. In the special case of $M=\Sph^d$, we have $A^{g_*}=g_*/2$, giving the values of $\sigma_1^{g_*}$ and $\sigma_2^{g_*}$ mentioned above.

Due to its connection to the celebrated Chern--Gauss--Bonnet formula in dimension $d=4$ \cite{Chang2002a} and its conformal properties \cite{Viaclovsky2000, Li2003}, the $\sigma_k^g$-curvature has been the subject of many studies over the last quarter century, in particular, in form of the $\sigma_k$-Yamabe problem, which can be formulated as follows. 
Find a metric $g$ conformal to a given metric $g_0$ such that the equation
\begin{equation}\label{eq:sigma2}
	\sigma_k^g = c
\end{equation} 
holds for some constant $c>0$. In case $k\geq 2$, this is a fully nonlinear equation for the conformal factor. As is common practice, \eqref{eq:sigma2} is considered under the constraints $\sigma_1^{g},\sigma_2^{g}>0$, which assure ellipticity of the resulting equation; we refer to  \cite{Viaclovsky2000} for more details.

For $k=1$ this is the famous Yamabe problem (in the positive case), which was solved in \cite{Yamabe1960, Trudinger1968, Aubin1976, Schoen1984}; see also the review \cite{Lee1987}. The $\sigma_2$-Yamabe problem was completely solved; see \cite{Gursky2007} for $d=3$, \cite{Chang2002, Chang2002a} for $d=4$, and \cite{Sheng2007} for $d\geq5$; see also the preliminary results in \cite{Li2003,Ge2005}. Some of these works also contain results about the $\sigma_k$-Yamabe problem, in particular when it is variational. For a review, we refer to \cite{Sheng2012}.

Analogously as for the usual Yamabe problem, in the resolution of the $\sigma_2$-Yamabe problem an important role is played by the functional $\mathcal F_2[g]$, defined for metrics $g$ on $M$ similarly as in the special case $M=\Sph^d$, and its infimum $Y_2(M,[g_0])$ over all metrics $g$ in the conformal class $[g_0]$ satisfying $\sigma_1^g>0$ and $\sigma_2^g>0$. In \cite{Sheng2007}, concerning $d\geq 5$, it is shown that the strict inequality $Y_2(M,[g_0]) < Y_2(\Sph^d,[g_*])$ implies the solvability of the $\sigma_2$-Yamabe problem, and then the strict inequality is shown whenever $(M,g_0)$ is not conformal to $(\Sph^d,g_*)$.

\subsubsection*{Background from stability of functional inequalities}
In recent years, there has been a huge interest in the problem of establishing stability estimates for functional and geometric inequalities. In the context of the Sobolev inequality in the homogeneous Sobolev space $\dot W^{1,2}(\R^d)$, the question for quantitative stability was first raised by Brezis and Lieb \cite{BREZIS198573}. They were looking for a `natural' way to bound the distance to the set of optimizers from above by the nonnegative difference between both sides of the Sobolev inequality. Bianchi and Egnell  \cite{Bianchi1991} gave an affirmative answer to this problem by showing that the $\dot W^{1,2}(\R^d)$-norm can be taken as the notion of distance and that the square of this distance can be controlled by the deficit in the Sobolev inequality. Their strategy is quite robust and has been applied to many other functional inequalities, including for instance the Sobolev inequality in $\dot W^{s,2}(\R^d)$, $s<d/2$; see \cite{Chen2013}.
	
This result has been extended to the Sobolev inequality in $\dot W^{1,p}(\R^d)$. The final answer, with the distance in $\dot W^{1,p}(\R^d)$ and with the optimal stability exponent, was obtained by Figalli and Zhang \cite{Figalli2022}, following earlier work in \cite{Cianchi2009,Figalli2015,Neumayer2019}. Interestingly, Figalli and Zhang found that the optimal stability exponent depends on $p$ and is given by $\max\{2,p\}$. This has the same reason as the appearance of the power $4$ of the $W^{1,4}$-norm in Theorem \ref{thm}.

Returning to the case $p=2$, we recall that the corresponding Sobolev inequality has the geometric meaning that among all unit-volume metrics on $\Sph^d$ that are conformal to the standard metric, the latter and its conformal images minimize the total scalar curvature. Stability for the corresponding Yamabe inequality on general closed Riemannian manifolds was shown by Engelstein, Neumayer, and Spolaor \cite{Engelstein2020}. Interestingly, they also showed the appearance of a stability exponent strictly larger than $2$. In the special case of $M=\Sph(R)\times\Sph^{d-1}$, the sharp stability exponent was found in \cite{Frank2023a} to be $4$ for the critical value of $R$; see also \cite{Frank2023} and, for another related model, \cite{Frank2024}. We emphasize that the reason for the appearance of a stability exponent different from $2$ in these cases is different from its appearance in the $p$-Sobolev inequality. In the setting of \cite{Frank2023a} the quartic behavior can be upgraded to a quadratic behavior away from a finite-dimensional subspace, as pointed out in \cite{BrDoSi}.

We think it would be interesting to investigate to which extent the stability for the Yamabe functional in \cite{Engelstein2020} has an analogue for the $\sigma_2$-curvature inequality on general closed Riemannian manifolds. This question might be related to studying the convergence rates of the $\sigma_2$-Yamabe flow. For the usual Yamabe flow, see \cite{Brendle2005,Brendle2007,Carlotto2015}.


\subsection{Strategy of the proof}\label{sec:strategy}
	
In this subsection we will describe our strategy and reduce the proof of Theorem \ref{thm} to the proof of two results, Propositions \ref{prop:glob2loc} and \ref{prop:loc} below.

Our overall strategy follows that of Bianchi and Egnell \cite{Bianchi1991} and consists of two steps. The first one is a global-to-local reduction, meaning that the inequality is valid for general functions, once it has been proved for functions close to the set of optimizers. This is the content of the upcoming Proposition \ref{prop:glob2loc}. The second step then is the proof of the inequality for functions close to the set of optimizers. This is accomplished in Proposition \ref{prop:loc}. 

The precise formulation of these results is the following.
	
\begin{proposition}\label{prop:glob2loc}
	Let $(u_j)\subset C^\infty(\Sph^d)$ be a sequence of positive functions with $\sigma_1(u_j)>0$ for all $j$ and satisfying, as $j\to\infty$, 
	$$
	F_2[u_j] \to S_d^{(2)}
	\qquad\text{and}\qquad
	\| u_j \|_{\frac{4d}{d-4}} \to \| 1 \|_{\frac{4d}{d-4}} \,.
	$$
	Then
	$$
	\inf_{\Psi} \| (u_j)_{\Psi} - 1 \|_{W^{1,4}(\Sph^d)} \to 0
	\qquad\text{as}\ j\to\infty \,.
	$$
\end{proposition}	

\begin{proposition}\label{prop:loc}
	There is a constant $c_d>0$ with the following property: Let $(u_j)\subset C^\infty(\Sph^d)$ be a sequence of positive functions with $\sigma_1(u_j)>0$ for all $j$, with $\| u_j \|_\frac{4d}{d-4} = \| 1\|_\frac{4d}{d-4}$ for all $j$ and with $\inf_{\Psi} \| (u_j)_{\Psi} - 1 \|_{W^{1,4}(\Sph^d)} \to 0$ as $j\to\infty$. Then
	$$
	\liminf_{j\to\infty} \frac{F_2[u_j]-S_d^{(2)}}{\inf_{\Psi} \left( \| (u_j)_{\Psi} - 1 \|_{W^{1,2}(\Sph^d)}^2 + \| (u_j)_{\Psi} - 1 \|_{W^{1,4}(\Sph^d)}^4 \right)} \geq c_d \,.
	$$
\end{proposition}	

Before discussing the difficulties in proving these two propositions, let us give the brief argument that shows that they imply our main stability theorem.

\begin{proof}[Proof of Theorem \ref{thm}]
	We argue by contradiction, assuming there is a sequence $(u_j)\subset C^\infty(\Sph^d)$ of positive functions with $\sigma_1(u_j)>0$ for all $j$ and
	\begin{equation}
		\label{eq:thmproofass}
		\frac{F_2[u_j]-S_d^{(2)}}{\inf_{\lambda,\Psi} \left( \| \lambda\, (u_j)_{\Psi} - 1 \|_{W^{1,2}(\Sph^d)}^2 + \| \lambda\, (u_j)_{\Psi} - 1 \|_{W^{1,4}(\Sph^d)}^4 \right)} \to 0 \,.
	\end{equation}
	By homogeneity we may assume that $\| u_j \|_\frac{4d}{d-4} = \| 1\|_\frac{4d}{d-4}$ for all $j$. Since
	$$
	\inf_{\lambda,\Psi} \left( \| \lambda\, (u_j)_{\Psi} - 1 \|_{W^{1,2}(\Sph^d)}^2 + \| \lambda\, (u_j)_{\Psi} - 1 \|_{W^{1,4}(\Sph^d)}^4 \right) \leq \| 1 \|_{W^{1,2}(\Sph^d)}^2 + \| 1 \|_{W^{1,4}(\Sph^d)}^4 \,,
	$$
	we deduce from \eqref{eq:thmproofass} that $F_2[u_j] \to S_d^{(2)}$ as $j\to\infty$. Thus, $\inf_{\Psi} \| (u_j)_{\Psi} - 1 \|_{W^{1,4}(\Sph^d)} \to 0$ as $j\to\infty$ by Proposition \ref{prop:glob2loc}. Therefore, we can apply Proposition \ref{prop:loc}, and, since
	\begin{align*}
		& \inf_{\lambda,\Psi} \left( \| \lambda\, (u_j)_{\Psi} - 1 \|_{W^{1,2}(\Sph^d)}^2 + \| \lambda\, (u_j)_{\Psi} - 1 \|_{W^{1,4}(\Sph^d)}^4 \right) \\
		& \leq \inf_{\Psi} \left( \| (u_j)_{\Psi} - 1 \|_{W^{1,2}(\Sph^d)}^2 + \| (u_j)_{\Psi} - 1 \|_{W^{1,4}(\Sph^d)}^4 \right),
	\end{align*}
	we obtain a contradiction to \eqref{eq:thmproofass}.
\end{proof}

Propositions \ref{prop:glob2loc} and \ref{prop:loc} will be proved in Sections \ref{sec:2} and \ref{sec:3}, respectively. Let us explain some of the ingredients that go into these proofs.

The global-to-local reduction in Proposition \ref{prop:glob2loc} is a compactness theorem for optimizing sequences. Bianchi and Egnell \cite{Bianchi1991}, as well as essentially all follow-up works, have deduced corresponding compactness theorems using Lions's method of concentration compactness \cite{Lions1985,Lions1985a} or variants thereof. We are not aware of any application of this method to $\mathcal F_2$- or related minimization problems. The obstacles we face are on the one hand the pointwise constraint $\sigma_1(u)>0$, which seems difficult to localize, and, on the other hand, the noncontrollability of $\sigma_1(u)$ through energy norms.

Our way of circumventing this problem is the use of the monotonicity of quotient functionals, due to Ge and Wang \cite{Ge2013}. This inequality essentially allows us to reduce the compactness for minimizers for $\mathcal F_2$ to the compactness of minimizers for $\mathcal F_1$, that is, for the usual Yamabe functional, to which Lions's method is applicable. This necessitates some technical work due to the different parametrizations $g=w^{4/(d-2)}g_*$ and $g=u^{8/(d-4)}g_*$ in the $\mathcal F_1$- and $\mathcal F_2$-minimization problems, respectively. We hope that this way of obtaining a compactness theorem is useful in other problems as well.

The local stability result in Proposition \ref{prop:loc} involves different notions of smallness, namely, in $W^{1,2}$ and $W^{1,4}$, and reflects a different behavior of the functional on functions that are small in either of the two senses. Bianchi and Egnell \cite{Bianchi1991}, and many derivative works, obtained their local stability theorems by a rather straightforward Taylor expansion. The works \cite{Figalli2015,Figalli2022} concerning the Sobolev inequality in $\dot W^{1,p}(\R^d)$ presented a breakthrough, where much more subtle Taylor expansions are used. However, in contrast to our work, they only consider one notion of smallness, namely, in $\dot W^{1,p}(\R^d)$.
	
The way we capture simultaneously both notions of smallness is through a decomposition into spherical harmonics. They are commonly used to effectively exploit orthogonality conditions \cite{Figalli2015,guerra2023sharp,Frank2024}, to facilitate a higher order expansion \cite{Frank2023, Frank2024,Koenig2024} on the sphere, or to make use of improved regularity for finite-dimensional projections \cite{Dolbeault2023}. In our context it is crucial that for spherical harmonics of a bounded degree the $W^{1,2}$- and $W^{1,4}$-norms are equivalent, while for spherical harmonics of diverging degree terms that are naively quartic can, in fact, behave like quadratic terms and are therefore relevant for the stability functional. 

Finally, we mention that recently there have been attempts to obtain explicit constants for functional inequalities in quantitative form; see, for instance, \cite{Bonforte2025,Bonforte2023,Dolbeault2023,BrDoSi}. It is conceivable that a modification of our proof also leads to an explicit stability constant, possibly even with the optimal dependence on the dimension as in \cite{Dolbeault2023}. We have refrained from exploring this since the necessary additional work would overshadow the conceptual ideas that are needed to deal with the new fully nonlinear setting.


\section{Global-to-local reduction}\label{sec:2}

Our goal in this section is to prove Proposition \ref{prop:glob2loc}. That is, we want to show that if $F_2[u]$ is close to $S_d^{(2)}$, then $\lambda (u)_\Psi$ is close to $1$ in $W^{1,4}(\Sph^d)$ for some $\lambda\in\R$ and some M\"obius transformation $\Psi$. (In fact, we will restrict ourselves to appropriately normalized $u$'s and show that we may take $\lambda=1$.) We emphasize that this is a qualitative assertion, as opposed to the quantitative assertion in our main theorem, Theorem \ref{thm}.

The key ingredient in the proof of Proposition \ref{prop:glob2loc} will be the fact that for any metric $g$ that is conformally equivalent to $g_*$ and has $\sigma_1^g>0$, one has
\begin{equation}
	\label{eq:gewang}
	\left( \frac{\mathcal F_2[g]}{S_d^{(2)}} \right)^\frac1{d-4} \geq \left( \frac{\mathcal F_1[g]}{S_d^{(1)}} \right)^\frac1{d-2} \geq 1 \,.
\end{equation}
Let us show how the inequality between $\mathcal F_1$ and $\mathcal F_2$ can be deduced from the works of Guan and Wang \cite{Guan2004} and Ge and Wang \cite{Ge2013}. Indeed, the above inequality appears in \cite[Theorem 1]{Guan2004} under the additional assumption $\sigma_2^g>0$. Thus,
$$
\inf_{\sigma_1(g)>0,\, \sigma_2(g)>0} \frac{\mathcal F_2[g]}{\left( \mathcal F_1[g] \right)^\frac{d-4}{d-2}} = \frac{S_d^{(2)}}{\left( S_d^{(1)} \right)^\frac{d-4}{d-2}} \,.
$$
Meanwhile, it follows from \cite[Theorem 1]{Ge2013} that
$$
\inf_{\sigma_1(g)>0,\, \sigma_2(g)>0} \frac{\mathcal F_2[g]}{\left( \mathcal F_1[g] \right)^\frac{d-4}{d-2}} = \inf_{\sigma_1(g)>0} \frac{\mathcal F_2[g]}{\left( \mathcal F_1[g] \right)^\frac{d-4}{d-2}} \,.
$$
Combining these two facts, we arrive at \eqref{eq:gewang}.

\begin{proof}[Proof of Proposition \ref{prop:glob2loc}]
	Let $(u_j)\subset C^\infty(\Sph^d)$ be a sequence of positive functions with $\sigma_1(u_j)>0$ for all $j$ and satisfying, as $j\to\infty$,
	$$
	F_2[u_j] \to F_2[1]
	\qquad\text{and}\qquad
	\| u_j \|_{\frac{4d}{d-4}} \to \| 1\|_{\frac{4d}{d-4}}\,.
	$$
	Our goal is to show that there is a sequence $(\Psi_j)$ of M\"obius transformations such that
	\begin{equation}
		\label{eq:glob2locgoal}
		(u_j)_{\Psi_j}\to 1
		\qquad\text{in}\ W^{1,4}(\Sph^d) \,.
	\end{equation}
	This will clearly imply the proposition. Also note that it suffices to prove this convergence for a subsequence (because if every subsequence has a further subsequence along which this convergence holds, then it holds for the full sequence). Moreover, by passing to normalized $u_j/\|u_j\|_{\frac{4d}{d-4}}$, we may assume $\|u_j\|_{\frac{4d}{d-4}}=\|1\|_{\frac{4d}{d-4}}$ without loss of generality.
	
	Applying inequality \eqref{eq:gewang} to the metrics $g_j:= u_j^\frac{8}{d-4}g_*$ and noting that $\mathcal F_2[g_j] = F_2[u_j] \to S^{(2)}_d$, we deduce that $\mathcal F_1[g_j] \to S^{(1)}_d$. Thus, if we define the positive functions $w_j$ on $\Sph^d$ by
	$$
	w_j^\frac{4}{d-2} = u_j^\frac 8{d-4} \,,
	$$
	then $F_1[w_j] = \mathcal F_1[g_j] \to S^{(1)}_d$. Moreover, note that
	$$
	\int_{\Sph^d} w_j^\frac{2d}{d-2}\,\mathrm d\omega = \int_{\Sph^d} u_j^\frac{4d}{d-4}\,\mathrm d\omega = |\Sph^d|
	\qquad\text{for all}\ j \,.
	$$
	It follows from Lions's theorem \cite{Lions1985,Lions1985a}, translated from $\R^d$ to $\Sph^d$, that, after passing to a subsequence, there is a sequence of M\"obius transformations $(\Psi_j)$ such that $[w_j]_{\Psi_j}\to 1$ in $W^{1,2}(\Sph^d)$. Let us set
	$$
	s_j := [w_j]_{\Psi_j} - 1
	\qquad\text{and}\qquad
	r_j := (s_j+1)^\frac{d-4}{2(d-2)} - 1 \,.
	$$
	Note that $r_j$ is well-defined since $s_j+1 = [w_j]_{\Psi_j}> 0$. It follows that
	$$
	(u_j)_{\Psi_j} = \left( [w_j]_{\Psi_j} \right)^\frac{d-4}{2(d-2)} = (1+s_j)^\frac{d-4}{2(d-2)} = 1 + r_j \,.
	$$
	We know that $s_j\to 0$ in $W^{1,2}(\Sph^d)$ and need to show that $r_j\to 0$ in $W^{1,4}(\Sph^d)$. To do so, we note that, by conformal invariance,
	\begin{equation}
		\label{eq:glob2locproof}
		S_d^{(2)} + o(1) = F_2[u_j] = F_2[(u_j)_{\Psi_j}] = |\Sph^d|^{-\frac{d-4}d} \int_{\Sph^d} e_2(1+r_j)\,\mathrm d\omega \,.
	\end{equation}
	We have
	$$
	e_2(1+r_j) = f_j +\frac{d(d-1)}{8} (1+r_j)^4 \,.
	$$
	with
	$$
	f_j := \left(\frac{4}{d-4}\right)^3 \left(\sigma_1(1+r_j)+\frac{1}{2} |\nabla r_j|^2+\frac{d-2}{2}\left(\frac{d-4}{4}\right)^2 (1+r_j)^2\right) |\nabla r_j|^2 \,.
	$$
	By the Sobolev inequality, we have $s_j\to 0$ in $L^\frac{2d}{d-2}(\Sph^d)$ and therefore also $r_j\to 0$ in $L^\frac{4d}{d-4}(\Sph^d)$. In particular, we have $(1+r_j)^4\to 1$ in $L^1(\Sph^d)$ and, consequently, by the explicit expression for $S_d^{(2)}$,
	$$
	|\Sph^d|^{-\frac{d-4}d} \int_{\Sph^d} \frac{d(d-1)}{8} (1+r_j)^4\,\mathrm d\omega = S^{(2)}_d + o(1) \,.
	$$
	Combined with \eqref{eq:glob2locproof}, we find that
	$$
	\int_{\Sph^d} f_j \,\mathrm d\omega = o(1) \,.
	$$
	Since $f_j$ is a sum of nonnegative terms, we deduce that $|\nabla r_j|\to 0$ in $L^4(\Sph^d)$, which proves \eqref{eq:glob2locgoal} along a subsequence, as desired.
\end{proof}


\section{Closeness in $W^{1,2}(\Sph^d)$ vs. closeness in $W^{1,4}(\Sph^d)$}

In the previous section we have shown that normalized optimizing sequences $(u_j)$ are relatively compact in $W^{1,4}(\Sph^d)$ up to M\"obius transformations. In this section we would like to bring them into a canonical form $(u_j)_{\Psi_j} = 1 + r_j$, where the remainders $r_j$ not only tend to zero in $W^{1,4}(\Sph^d)$ but also satisfy some (almost-)orthogonality conditions. The latter come from the normalization and an optimal choice of $\Psi_j$.

More precisely, our goal in this section is to prove the following result.

\begin{proposition} \label{prop:qualstab} 
	Let $(u_j)\subset W^{1,4}(\Sph^d)$ with $\| u_j \|_{\frac{4d}{d-4}} = \|1\|_{\frac{4d}{d-4}}$ for all $j$ and $\inf_\Psi \|(u_j)_\Psi -1\|_{W^{1,4}}\to 0$ as $j\to\infty$. Then there is a sequence $(\Psi_j)$ of M\"obius transformations such that
	$$
	r_j := (u_j)_{\Psi_j} - 1
	$$
	satisfies, for all sufficiently large $j$,
	\begin{equation}
		\label{eq:distancebounds}
			\| r_j \|_{W^{1,2}} = \inf_\Psi \|(u_j)_\Psi -1\|_{W^{1,2}}
		\qquad\text{and}\qquad
		\| r_j \|_{W^{1,4}} \lesssim \inf_\Psi \|(u_j)_\Psi -1\|_{W^{1,4}}
	\end{equation}
	as well as
	\begin{equation}\label{eq:ortho}
		\left|	\int_{\mathbb S^d} r_j \, \mathrm d \omega \right| \lesssim \|r_j\|^{2}_2+\|r_j\|^{\frac{4d}{d-4}}_{\frac{4d}{d-4}}  \qquad	\text{and} \qquad
		\left| \int_{\mathbb S^d} \omega_i \, r_j \, \mathrm d \omega \right| \lesssim \|r_j\|^2_{W^{1,2}} \,, \ i=1,\dots,d+1\,.
	\end{equation}
\end{proposition}

Here and below we use the symbol $\lesssim$ to denote that the left side is bounded by the right side times a constant that only depends on $d$.

Note that in the almost-orthogonality conditions \eqref{eq:ortho} we bound linear forms of $r_j$ in terms of norms that vanish faster than linearly. This should be interpreted as saying that the linear forms are `almost' zero.

The proof of Proposition \ref{prop:qualstab} will be given in Subsection \ref{sec:qualstabproof}, after we have established some auxiliary results in Subsections \ref{sec:mobius}, \ref{sec:distances}, and \ref{sec:almostortho}. The discussion in this section is valid for general functions in $W^{1,4}(\Sph^d)$, not necessarily smooth nor positive.


\subsection{Preliminaries about M\"obius transformations}\label{sec:mobius}

This subsection contains some preparations for the proof of Proposition \ref{prop:qualstab}. The first result says that, while the $W^{1,4}$-norm is not invariant under M\"obius transformations, it is so up to a constant that only depends on $d$.

\begin{lemma}\label{lem:confbdd}
	There is a constant $C$ such that for all $u\in W^{1,4}(\Sph^d)$,
\begin{equation*}
		\sup_\Psi \| (u)_{\Psi}\|_{W^{1,4}}\leq C \| u\|_{W^{1,4}}\,.
	\end{equation*}
\end{lemma}

For the proof of this lemma, we will use an explicit parametrization of M\"obius transformations by elements $\xi$ in $B_1(0)$, the unit ball in $\R^{d+1}$. This is well known. One way of obtaining this is through the image under stereographic projection of the transformations on $\R^d\cup\{\infty\}$ generated by the Euclidean group, scaling, and inversion; see, e.g., \cite[Proof of Lemma 8]{Frank2023} or \cite[Subsection 2.2]{guerra2023sharp}. More specifically, we set
\begin{equation*}
	\Psi_\xi(\omega)
	:= \frac{(1-|\xi|^2)\omega -2(1-\xi\cdot\omega)\xi}{1-2\xi\cdot\omega + |\xi|^2}
	\,,\qquad \omega \in \mathbb S^d \,.
\end{equation*}
A short computation shows that
\begin{equation}
	\label{eq:psijac}
	(J_{\Psi_\xi}(\omega))^\frac 1d = \frac{1-|\xi|^2}{1-2\xi\cdot\omega + |\xi|^2}
\end{equation}
as well as 
\begin{equation*}
	\Psi_\xi^{-1}(\omega) = \Psi_{-\xi}(\omega)
	= \frac{(1-|\xi|^2)\omega + 2(1+\xi\cdot\omega)\xi}{1+2\xi\cdot\omega + |\xi|^2}
	\,,\qquad \omega \in \mathbb S^d \,.
\end{equation*}
We note that the above expressions are slightly different from those in \cite{Frank2023}. Setting $\zeta = 2\xi/(1+|\xi|^2)$ in our formulas, we obtain those in \cite{Frank2023}, where $\zeta$ plays the role of what is denoted $\xi$ there. Note that $\xi\mapsto\zeta$ is a bijection of $B_1(0)$.

\begin{proof}
	We note that
	$$
	\| (u)_{\Psi}\|_4 \lesssim \| (u)_{\Psi}\|_\frac{4d}{d-4} = \| u \|_\frac{4d}{d-4} \lesssim \| u \|_{W^{1,4}} \,,
	$$
	so we only need to show that
	$$
	\sup_\Psi \| \nabla (u)_{\Psi}\|_4 \lesssim \| u \|_{W^{1,4}} \,.
	$$
	We write $(u)_{\Psi} = J_\Psi^\frac{d-4}{4d} (u\circ\Psi)$ and bound
	$$
	\|\nabla (u)_{\Psi}\|_{4} \leq \left\| \left(\nabla J_\Psi^\frac{d-4}{4d}\right) (u\circ\Psi) \right\|_{4} + \left\|J_\Psi^\frac{d-4}{4d} (d\Psi)^{\rm T}((\nabla u)\circ \Psi) \right\|_{4} \,.
	$$
	(Here $(d\Psi_\omega)^{\rm T}:T_{\Psi(\omega)}\Sph^d \to T_\omega\Sph^d$ is the adjoint of the map $d\Psi_\omega:T_\omega\Sph^d \to T_{\Psi(\omega)}\Sph^d$ with respect to the given inner products on these spaces.) It follows from the fact that $\Psi$ is a conformal transformation that
	$$
	|(d\Psi)^{\rm T}((\nabla u)\circ \Psi_\xi)| = J_\Psi^\frac{1}{d} |(\nabla u)\circ \Psi_\xi| \,,
	$$
	and therefore, by a change of variables,
	$$
	\left\| J_\Psi^\frac{d-4}{4d} (d\Psi)^{\rm T}((\nabla u)\circ \Psi_\xi) \right\|_{4} = \left\| J_\Psi^\frac{1}{4} |(\nabla u)\circ \Psi_\xi | \right\|_{4} = \left\|\nabla u \right\|_4 \,.
	$$

	Thus, to complete the proof of the lemma, it remains to prove
	$$
	\sup_\Psi \left\| \left(\nabla J_\Psi^\frac{d-4}{4d}\right) (u\circ\Psi) \right\|_{4} \lesssim \| u \|_{W^{1,4}} \,.
	$$
	To this end, we parametrize $\Psi$ as $\Psi_\xi$, so by \eqref{eq:psijac},
	\begin{align*}
		\left|\nabla J_\Psi^\frac{d-4}{4d}(\omega)\right| & = \frac{d-4}2 (1-|\xi|^2)^\frac{d-4}4 (1-2\xi\cdot\omega+|\xi|^2)^{-\frac d4} \sqrt{|\xi|^2 -(\xi\cdot\omega)^2} \\
		& = \frac{d-4}2 \, \frac{\sqrt{|\xi|^2 - (\xi\cdot\omega)^2}}{1-|\xi|^2} \, 
		J_\Psi^\frac{1}{4}(\omega) \,.
	\end{align*}
	By a change of variables,
	\begin{align*}
		\left\| \left(\nabla J_\Psi^\frac{d-4}{4d}\right) (u\circ\Psi_\xi) \right\|_{4} & = \frac{d-4}2 \left\| \frac{\sqrt{|\xi|^2 - (\xi\cdot\omega)^2}}{1-|\xi|^2} \, J_\Psi^\frac{1}{4} \, (u\circ\Psi_\xi) \right\|_4 \\
		& = \frac{d-4}2 \left\| \frac{\sqrt{|\xi|^2 - (\xi\cdot \Psi_\xi^{-1}(\omega))^2}}{1-|\xi|^2} \,  u \, \right\|_4 \,.
	\end{align*}
	Using the explicit form of $\Psi_\xi^{-1}$, we find 
	\begin{align*}
		\frac{\sqrt{|\xi|^2 - (\xi\cdot \Psi_\xi^{-1}(\omega))^2}}{1-|\xi|^2} & = \frac{\sqrt{|\xi|^2 - (\xi\cdot\omega)^2}}{1+2\omega\cdot\xi + |\xi|^2}
		= \frac{\sqrt{|\omega + \xi|^2 - (1+\xi\cdot\omega)^2}}{|\omega + \xi|^2}
		\leq \frac{1}{|\omega+\xi|} \leq \frac{1}{|\omega +\xi/|\xi||} \,.
	\end{align*}
	Thus,
	$$
	\sup_\Psi \left\| \left(\nabla J_\Psi^\frac{1}{4}\right) (u\circ\Psi) \right\|_{4} \lesssim \sup_{\omega_0\in\Sph^d} \| |\omega-\omega_0|^{-1} u \|_4 \,,
	$$
and consequently the lemma will follow if we prove the bound
	$$
	\sup_{\omega_0\in\Sph^d} \| |\omega-\omega_0|^{-1} u \|_4 \lesssim \| u \|_{W^{1,4}} \,.
	$$
	
	To do so, by rotation invariance we may assume that $\omega_0 = e_{d+1}$ is the north pole. We choose smooth functions $\chi_0, \chi_1$ on $\Sph^d$ with $\chi_0^4 + \chi_1^4 =1$ such that $\chi_0$ is equal to $1$ in a neighborhood of $e_{d+1}$ and vanishes outside another, larger neighborhood of $e_{d+1}$. We set $u_j := \chi_j u$ for $j=0,1$. Since $|\omega - e_{d+1}|^{-1}$ is nonsingular on the support of $\chi_1$, we have
	$$
	\| |\omega - e_{d+1}|^{-1} u_1 \|_4 \lesssim \| u_1 \|_4 \leq \| u_1 \|_{W^{1,4}} \,.
	$$
	To control $u_0$, we use Hardy's inequality on $\R^d$, $d>4$, which says that
	$$
	\int_{\R^d} \frac{|v|^4}{|x|^4}\,\mathrm dx \lesssim \int_{\R^d} |\nabla v|^4\,\mathrm dx \,.
	$$
	By going to local coordinates, we deduce that
	$$
	\| |\omega - e_{d+1}|^{-1} u_0 \|_4 \lesssim \| \nabla u_0 \|_4 \leq \| u_0 \|_{W^{1,4}} \,.
	$$
	The claimed bound now follows from the fact that
	$$
	\| u_0 \|_{W^{1,4}}^4 + \| u_1 \|_{W^{1,4}}^4 \lesssim \| u \|_{W^{1,4}}^4 \,,
	$$
	which is easy to see. This completes the proof of the lemma.
\end{proof}

In contrast to the critical case, the singularity of conformal images of $u$ under $\Psi_{\xi}$ as $|\xi|\to 1$ is negligible for subcritical exponents, which can be observed in the next lemma, at least in the special case $u=1$.

\begin{lemma} \label{lem:subcritical}
We have
\begin{equation*}
	\lim_{|\xi|\to 1}\|(1)_{\Psi_{\xi}}\|_q = 0 \quad\text{for}\quad q<\frac{4d}{d-4}
	\qquad\text{and}\qquad
	\lim_{|\xi|\to 1}\|\nabla (1)_{\Psi_{\xi}}\|_p = 0 \quad\text{for}\quad p<4 \,.
\end{equation*}
\end{lemma}

\begin{proof}
	Formula \eqref{eq:psijac} gives an explicit expression for $(1)_{\Psi_\xi}$ and its gradient; see also the previous proof. Given these explicit expressions, the assertion is a consequence of the elementary estimate
	\begin{equation}
		\label{eq:elementaryintegral}
		\int_{\Sph^d} \frac{\mathrm d\omega}{|\omega-\xi|^\alpha} \lesssim (1-|\xi|^2)^{-\alpha+d}
	\end{equation} for $\alpha>d$.
	Let us quickly justify this bound. The integral is equal to $|\Sph^{d-1}|$ times
	$$
	\int_0^\pi \frac{\sin^{d-1}\theta}{(1- 2|\xi|\cos\theta + |\xi|^2)^{\alpha/2}} \,\mathrm d\theta = \int_{-1}^1 \frac{(1-t^2)^{(d-2)/2}}{(1-2|\xi| t + |\xi|^2)^{\alpha/2}}\,\mathrm d t \,.
	$$
	We may clearly assume that $|\xi|\geq 1/2$. Then we bound $1-2|\xi| t + |\xi|^2$ in the denominator from below by $\gtrsim 1-t$ if $1-t\geq (1-|\xi|)^2$ and by $\gtrsim (1-|\xi|)^2$ if $1-t<(1-|\xi|)^2$, and perform the resulting integrals. This leads to \eqref{eq:elementaryintegral}.
\end{proof}

\begin{lemma}\label{lem:mobcomp}
	Let $u\in W^{1,4}(\Sph^d)$. Then for all $p\leq 4$,
	$$
	\liminf_{|\xi|\to 1} \|(u)_{\Psi_\xi}-1\|_{W^{1,p}} \geq \| 1 \|_{W^{1,p}} \,.
	$$
\end{lemma}

\begin{proof}
	Fix $p\leq 4$ and assume to the contrary that there is a sequence $(\xi_j)\subset B_1(0)$ such that $|\xi_j|\to 1$ and such that
	$$
	\lim_{j\to \infty} \|(u)_{\Psi_{\xi_j}}-1\|_{W^{1,p}} 
	\qquad\text{exists and is}\ < \| 1 \|_{W^{1,p}} \,.
	$$
	Passing to a subsequence, if necessary, we may assume that there is an $\omega\in\Sph^d$ such that $\xi_j\to\omega$. Then the computations in the previous proof show that
	$$
	\nabla (u)_{\Psi_{\xi_j}} \to 0
	\qquad\text{and}\qquad
	(u)_{\Psi_{\xi_j}} \to 0 \qquad\text{pointwise a.e.\,in}\ \Sph^d \setminus\{\omega\}\,.
	$$
	It follows from Fatou's lemma that
	$$
	\lim_{j\to \infty} \|(u)_{\Psi_{\xi_j}}-1\|_{W^{1,p}} \geq \| 1 \|_{W^{1,p}} \,.
	$$
	This is a contradiction.
\end{proof}


\subsection{Comparability of distances}\label{sec:distances}

We recall that the goal of this section is to find, for a given normalized $u$ with $\inf_{\Psi'} \| (u)_{\Psi'} - 1 \|_{W^{1,4}}$ small, a representation $(u)_\Psi = 1+r$ where $r$ is small in $W^{1,4}$ and satisfies almost-orthogonality conditions. These almost-orthogonality conditions are relatively easy to derive when $\Psi$ is chosen to minimize not the $W^{1,4}$- but the $W^{1,2}$-distance $\inf_{\Psi''} \| (u)_{\Psi''} - 1 \|_{W^{1,2}}$, as shown in the next subsection. However, when $\Psi$ is chosen in this manner, it is not obvious that the corresponding $r$ is small in $W^{1,4}$. That it is small in $W^{1,4}$ is the content of the following result, which has some similarities with \cite[Proposition 4.1]{Figalli2015}.

\begin{lemma}\label{lem:att} 
	There is a constant $C$ such that for any $u\in W^{1,4}(\Sph^d)$ with $\inf_{\Psi} \| (u)_{\Psi} - 1 \|_{W^{1,2}}<\| 1 \|_{W^{1,2}}$, there is a M\"obius transformation $\Psi'$ such that
	$$
	\| (u)_{\Psi'} - 1 \|_{W^{1,2}} = \inf_{\Psi} \| (u)_{\Psi} - 1 \|_{W^{1,2}}
	$$
	and
	\begin{equation}
		\label{eq:bdd}
		\| (u)_{\Psi'} - 1 \|_{W^{1,4}} \leq C \inf_{\Psi} \| (u)_{\Psi} - 1 \|_{W^{1,4}} \,.
	\end{equation}
\end{lemma}

\begin{proof}
	\emph{Step 1.} We begin by proving that for all $u\in W^{1,4}(\Sph^d)$ with $\inf_\Psi \| (u)_\Psi - 1\|_{W^{1,2}}<\| 1 \|_{W^{1,2}}$, the latter infimum is attained.
	
	We use the parametrization of M\"obius transformations $\Psi_\xi$ with $\xi\in B_1(0)$ as in Lemma \ref{lem:mobcomp}. It is easy to see that $\xi\mapsto (u)_{\Psi_\xi}$ is continuous as a map from $B_1(0)$ to $W^{1,2}(\Sph^d)$. Consequently, $\xi\mapsto \| (u)_{\Psi_\xi} - 1\|_{W^{1,2}}$ is continuous in $B_1(0)$. By Lemma \ref{lem:mobcomp} its limit inferior for $|\xi|\to 1$ is $\geq \| 1 \|_{W^{1,2}}$, while by assumption it assumes a value $<\| 1 \|_{W^{1,2}}$ in the interior. Therefore, by compactness the infimum is attained in $B_1(0)$.
	
	\medskip
	
	\emph{Step 2.}
	To prove \eqref{eq:bdd}, we assume by contradiction that there is a sequence $(u_j)\subset W^{1,4}(\Sph^d)$ with $\inf_\Psi \|(u_j)_\Psi - 1 \|_{W^{1,2}} < \|1\|_{W^{1,2}}$ for all $j$ as well as a sequence $(\Psi_j')$ of M\"obius transformations such that
	$$
	\|(u_j)_{\Psi_j'} - 1 \|_{W^{1,2}} = \inf_\Psi \|(u_j)_\Psi - 1 \|_{W^{1,2}}
	\qquad\text{and}\qquad
	\|(u_j)_{\Psi_j'} - 1 \|_{W^{1,4}} > j \inf_\Psi \|(u_j)_\Psi - 1 \|_{W^{1,4}} \,.
	$$
	
	By definition of the infimum $\inf_\Psi \|(u_j)_\Psi - 1 \|_{W^{1,4}}$, there is a sequence $(\Psi_j'')$ of M\"obius transformations such that
	\begin{equation}\label{eq:bdd1}
		\|(u_j)_{\Psi_j'} - 1 \|_{W^{1,4}} > j \|(u_j)_{\Psi_j''} - 1 \|_{W^{1,4}} \,.
	\end{equation}
	
	Thanks to the invariance (up to constants) of the $W^{1,4}$-norm under M\"obius transformations (Lemma \ref{lem:confbdd}), we deduce that
	\begin{equation}\label{eq:bdd2}
		\|(u_j)_{\Psi'_j}-1\|_{W^{1,4}} \gtrsim j 	\|(u_j)_{\Psi'_j}-(1)_{(\Psi''_j)^{-1}\circ \Psi'_j}\|_{W^{1,4}} \,. 
	\end{equation}
	
	Employing \eqref{eq:bdd2} and the triangle inequality, we observe that
	$$
	v_j\coloneqq \frac{(u_j)_{\Psi'_j}-1}{\|(u_j)_{\Psi'_j}-1\|_{W^{1,4}}}\,, \qquad\tilde v_j\coloneqq \frac{(u_j)_{\Psi'_j}-(1)_{(\Psi''_j)^{-1}\circ \Psi'_j}}{\|(u_j)_{\Psi'_j}-1\|_{W^{1,4}}}\,, \qquad\Delta_j\coloneqq v_j-\tilde v_j 
	$$ 
	satisfy, as $j\to\infty$,
	\begin{equation}\label{eq:phiconv}
		\|v_j\|_{W^{1,4}}\to 1\,,\qquad \|\tilde v_j\|_{W^{1,4}}\to 0\,, \qquad\|\Delta_j\|_{W^{1,4}}\to 1 \,.
	\end{equation} Applying the Sobolev inequality and \eqref{eq:bdd1}, we find 
	$$
	\|v_j\|_{W^{1,2}} = \frac{\inf_\Psi \|(u_j)_\Psi - 1 \|_{W^{1,2}}}{\|(u_j)_{\Psi'_j}-1\|_{W^{1,4}}} 
	\leq \frac{\|(u_j)_{\Psi''_j}-1\|_{W^{1,2}}}{\|(u_j)_{\Psi'_j}-1\|_{W^{1,4}}}\lesssim \frac{\|(u_j)_{\Psi''_j}-1\|_{W^{1,4}}}{\|(u_j)_{\Psi'_j}-1\|_{W^{1,4}}}\leq \frac{1}{j}\to 0\,,
	$$ 
	and thus
	\begin{equation*}
		\|\tilde v_j\|^2_{W^{1,2}}+2\langle \tilde v_j, \Delta_j\rangle_{W^{1,2}}+\|\Delta_j\|^2_{W^{1,2}}=	\| v_j\|^2_{W^{1,2}}\to 0 
	\end{equation*} as $j\to\infty$. By \eqref{eq:phiconv}, we have $\|\tilde v_j\|^2_{W^{1,2}}+2\langle \tilde v_j, \Delta_j\rangle_{W^{1,2}} \to 0$ and therefore
	$$
	\|\Delta_j\|^2_{W^{1,2}} \to 0 \,.
	$$
	
	Up to the factor of $\|(u_j)_{\Psi'_j}-1\|_{W^{1,4}}^{-1}$, the function $\Delta_j$ is given by $(1)_{(\Psi''_j)^{-1}\circ \Psi'_j} - 1$. We can think of the latter function as being a function of a variable $\xi_j''\in B_1(0)$ parametrizing the M\"obius transformation $(\Psi''_j)^{-1}\circ \Psi'_j$. We distinguish between two behaviors of the sequence $(\xi_j'')$.
	
	If $\liminf_{j\to\infty}|\xi_j''|<1$, then along the corresponding subsequence all norms of such functions of $\xi_j''$ are equivalent. As a consequence, the properties $\|\Delta_j\|_{W^{1,4}}\to 1$ and $\|\Delta_j\|^2_{W^{1,2}} \to 0$ contradict each other.
	
	If $\lim_{j\to\infty}|\xi_j''|=1$, then Lemma \ref{lem:subcritical} implies that $$\|(1)_{\Psi_{\xi_j''}}-1\|_{W^{1,2}}\to \|1\|_{W^{1,2}}\gtrsim 1\,.$$ 
	However, we know by Lemma \ref{lem:confbdd} that
	$$\|(1)_{\Psi_{\xi_j''}}-1\|_{W^{1,4}}\lesssim 1\,,$$ which leads to a contradiction as 
	$$\frac{\|(1)_{\Psi_{\xi_j''}}-1\|_{W^{1,2}}}{\|(1)_{\Psi_{\xi_j''}}-1\|_{W^{1,4}}}=\frac{\|\Delta_j\|_{W^{1,2}}}{\|\Delta_j\|_{W^{1,4}}}\to 0\,.$$
	
	Both cases together prove the second part of the proposition.
\end{proof}


\subsection{Almost-orthogonality conditions}\label{sec:almostortho}

We show that choosing the M\"obius transformation $\Psi$ optimal in the $W^{1,2}$-sense yields almost-orthogonality conditions for $r=(u)_\Psi -1$.

\begin{lemma}\label{lem:almostortho}
	Let $u\in W^{1,2}(\Sph^d)$ be such that $\inf_\Psi \| (u)_\Psi - 1 \|_{W^{1,2}}$ is attained at $\Psi=\operatorname{id}$. Then $r:= u-1$ satisfies
	$$
	\left| \int_{\Sph^d} \omega_i \, r \,\mathrm d\omega \right| \lesssim \| r \|_{W^{1,2}}^2 \,,
	\qquad i = 1,\ldots, d+1 \,.
	$$
\end{lemma}

\begin{proof}
	In the proof we may assume that $u$ is smooth (which is the only case needed later in this paper), for the more general assertion then follows by approximation. By assumption,
	$$
	\xi \mapsto \| (u)_{\Psi_\xi} - 1 \|_{W^{1,2}}^2
	$$
	assumes its infimum at $\xi=0$. Thus, its gradient with respect to $\xi$ vanishes there, which means that
	$$
	\left\langle r, \partial_{\xi_i}|_{\xi=0} (u)_{\Psi_\xi} \right\rangle_{W^{1,2}} = 0 \,,
	\qquad i = 1,\ldots, d+1 \,.
	$$
	
	In the following we fix $i\in\{1,\ldots,d+1\}$ and look for a more explicit form of the corresponding orthogonality condition. We have $(u)_{\Psi_\xi} = J_{\Psi_\xi}^\frac{d-4}{4d} (u\circ\Psi_\xi)$, so
	$$
	\partial_{\xi_i}|_{\xi=0} (u)_{\Psi_\xi} = \tfrac{d-4}4 \left( \partial_{\xi_i}|_{\xi=0} J_{\Psi_\xi}^\frac{1}{d} \right) \ u
	+ \partial_{\xi_i}|_{\xi=0} (u\circ\Psi_\xi) \,.
	$$
	Since
	$$
	J_{\Psi_\xi}(\omega)^\frac1d = 1+ 2\xi\cdot\omega + \mathcal O(|\xi|^2) \qquad \text{and}\qquad
	\Psi_\xi(\omega) = \omega - 2 (\xi-\xi\cdot\omega\,\omega) + \mathcal O(|\xi|^2) \,,
	$$we find
	$$
	\partial_{\xi_i}|_{\xi=0} (u)_{\Psi_\xi} = \tfrac{d-4}2 \,\omega_i \, u
	- 2 (e_i - \omega_i\,\omega)\cdot\nabla u 
	= \tfrac{d-4}2 \,\omega_i \, (1+r)
	- 2 (e_i - \omega_i\,\omega)\cdot\nabla r \,.
	$$
	Therefore, the $i$-th orthogonality condition becomes
	$$
	\left\langle r, \tfrac{d-4}2 \,\omega_i \, (1+r)
	- 2 (e_i - \omega_i\,\omega)\cdot\nabla r \right\rangle_{W^{1,2}} = 0 \,,
	$$
	or, equivalently,
	\begin{align}
		\label{eq:almostorthoproof}
			\left\langle r, \,\omega_i \right\rangle_{W^{1,2}}
		= \left\langle r, \tfrac4{d-4} (e_i - \omega_i\,\omega)\cdot\nabla r-\omega_i \, r   \right\rangle_{W^{1,2}} \,.
	\end{align}
	Note that the left side is linear in $r$, while the right side is quadratic in $r$. Using $-\Delta \omega_i = d \omega_i$ we see that the left side is equal to
	\begin{equation}
		\label{eq:almostorthoproof1}
		\left\langle r, \,\omega_i \right\rangle_{W^{1,2}} = (d+1) \int_{\Sph^d} \omega_i \, r \,\mathrm d\omega \,.
	\end{equation}
	Thus, the lemma will follow if we can bound the right side of \eqref{eq:almostorthoproof} by $\| r \|_{W^{1,2}}^2$. This is clear for the term containing $\omega_i \, r$. An argument is needed, however, for the term containing $(e_i - \omega_i\,\omega)\cdot\nabla r$. Since the $W^{1,2}$-norm involves a derivative, it seems like two derivatives of $r$ are needed, which is why we assumed that $u$ is smooth. We show now, however, that it can be controlled by a single derivative. The idea will be to integrate by parts. Indeed, the following lemma shows that
	$$
	\left| \nabla r \cdot \nabla \left(\left(e_i - \omega_i\,\omega \right) \cdot\nabla r \right) - \tfrac12 \left(e_i - \omega_i\,\omega \right)\cdot\nabla (|\nabla r|^2) \right| \lesssim |\nabla r|^2 \,.
	$$
	Since
	$$
	\int_{\Sph^d} \left(e_i - \omega_i \,\omega \right)\cdot\nabla (|\nabla r|^2) \,\mathrm d\omega = d \int_{\Sph^d} \omega_i |\nabla r|^2\,\mathrm d\omega \,,
	$$
	which is controlled by $\| \nabla r\|_2^2$, we deduce that
	$$
\left| \left\langle r, \tfrac4{d-4} (e_i - \omega_i\,\omega)\cdot\nabla r-  \omega_i \, r  \right\rangle_{W^{1,2}} \right| \lesssim \| r \|_{W^{1,2}}^2 \,.
	$$
	Therefore, \eqref{eq:almostorthoproof} and \eqref{eq:almostorthoproof1} imply the assertion of the lemma.
\end{proof}

The following simple observation was used in the proof of the previous lemma.

\begin{lemma}\label{lem:productrule}
	Let $X$ and $Y$ be (tangent) vector fields on $\Sph^d$ with $X$ being a gradient field. Then pointwise on $\Sph^d$,
	$$
	\left| X \cdot \nabla(X\cdot Y) - \tfrac12 Y \cdot \nabla ( |X|^2) \right| \lesssim |X|^2 (|\nabla Y| +  |Y| ) \,.
	$$ 
\end{lemma}

The main point of this inequality is that on the right side we do not have derivatives of $X$.

\begin{proof}
	In a smooth local orthonormal frame $(E_n)$ we compute, writing $X_n:= E_n\cdot X$, $Y_n:= E_n\cdot Y_n$,
	$$
	X \cdot \nabla(X\cdot Y) = \sum_n X_n \nabla_{E_n} (X\cdot Y) = \sum_n X_n \left( \nabla_{E_n} X \cdot Y + X \cdot \nabla_{E_n} Y \right).
	$$
	Since $\nabla_{E_n} Z = \sum_m (E_n Z_m) E_m + \sum_{m,k} \Gamma_{nm}^k Z_m E_k$ for a vector field $Z$ on $\mathbb S^d$ with the Christoffel symbols $\Gamma_{nm}^k$, we obtain
	$$
	X \cdot \nabla(X\cdot Y) = \sum_{n,m} X_n \left( (E_n X_m) Y_m + (E_n Y_m) X_m \right) + \sum_{n,m,k} \Gamma_{nm}^k X_n \left( X_m Y_k + Y_m X_k \right).
	$$
	We now use the assumption that $X=\nabla f$ is a gradient field. Then $X_n = E_n f$, and consequently
	\begin{align*}
		X_n (E_nX_m) & = (E_n f) (E_n E_m f) = \tfrac12 E_m ( (E_n f)^2) + (E_n f) ([E_n,E_m]f) \\
		& = \tfrac12 E_m ( (E_n f)^2) + \sum_k c_{nm}^k (E_n f) (E_k f) \,,
	\end{align*}
	where $[E_n,E_m] = \sum_k c_{nm}^k E_k$ for some $c_{nm}^k$. To summarize, for gradient fields $X$ we have shown that
	\begin{align*}
		X \cdot \nabla(X\cdot Y) & = \tfrac12 Y \cdot \nabla ( |X|^2) +  \sum_{n,m} (E_n Y_m) X_n X_m \\
		& \quad + \sum_{n,m,k} c_{nm}^k X_n X_k Y_m
		+ \sum_{n,m,k} \Gamma_{nm}^k X_n \left( X_m Y_k + Y_m X_k \right).
	\end{align*}
	This implies the assertion.
\end{proof}


\subsection{Proof of Proposition \ref{prop:qualstab}}\label{sec:qualstabproof}
We are now in position to prove the main result of this section.

\begin{proof}[Proof of Proposition \ref{prop:qualstab}]
	Since $\| \cdot \|_{W^{1,2}} \lesssim \| \cdot \|_{W^{1,4}}$ and since $\inf_\Psi \|(u_j)_\Psi -1 \|_{W^{1,4}}\to 0$, we may assume, after discarding finitely many $j$, that $\inf_\Psi \|(u_j)_\Psi -1 \|_{W^{1,2}}<\| 1\|_{W^{1,2}}$. Consequently, by Lemma \ref{lem:att}, for every $j$ there is a M\"obius transformation $\Psi_j$ such that
	$$
	\| (u_j)_{\Psi_j} - 1 \|_{W^{1,2}} =  \inf_\Psi \|(u_j)_\Psi -1 \|_{W^{1,2}} \,.
	$$
	This proves the equality in \eqref{eq:distancebounds}, and the inequality there follows from Lemma \ref{lem:att} as well.
	
	Using the elementary inequality, valid for all $\rho\in\R$,
	$$
	\left| | 1+ \rho |^\frac{4d}{d-4} - 1 - \frac{4d}{d-4} \rho \right| \lesssim \rho^2 + |\rho|^\frac{4d}{d-4} \,,
	$$
	we deduce that
	\begin{equation}
		\label{eq:orthoproof1}
		\left| \int_{\Sph^d} |1+r_j|^\frac{4d}{d-4} \,\mathrm d\omega - |\Sph^d| - \frac{4d}{d-4} \int_{\Sph^d} r_j \,\mathrm d\omega \right| \lesssim \|r_j\|_2^2 + \|r_j\|_\frac{4d}{d-4}^\frac{4d}{d-4} \,.
	\end{equation}
	Since
	$$
	\int_{\Sph^d} |1+r_j|^\frac{4d}{d-4} \,\mathrm d\omega = \int_{\Sph^d} |(u_j)_{\Psi_j}|^\frac{4d}{d-4} \,\mathrm d\omega = \int_{\Sph^d} |u_j|^\frac{4d}{d-4} \,\mathrm d\omega = |\Sph^d| \,,
	$$
	the first two terms on the left side of \eqref{eq:orthoproof1} cancel, and we obtain the first almost-orthogonality property in \eqref{eq:ortho}.
	
	The remaining almost-orthogonality properties in \eqref{eq:ortho} follow from Lemma \ref{lem:almostortho} applied to $u = (u_j)_{\Psi_j}$.
\end{proof}


\section{Local analysis}\label{sec:3}

In this section we shall prove Proposition \ref{prop:loc}, which states the validity of the stability inequality for functions close to the set of optimizers.

\subsection{Outline of the proof}

Throughout this section we assume that $(u_j)\subset W^{1,4}(\Sph^d)$ satisfies $\|u_j\|_{\frac{4d}{d-4}} = \|1\|_\frac{4d}{d-4}$ and $\inf_\Psi \| (u_j)_\Psi -1 \|_{W^{1,4}}\to 0$. According to Proposition \ref{prop:qualstab}, there are M\"obius transformations $(\Psi_j)$ such that we can write
$$
(u_j)_{\Psi_j} = 1 + r_j \,,
$$
where $r_j$ satisfies
$$
\| r_j \|_{W^{1,4}} \to 0
$$
as well as
\begin{equation}
	\label{eq:almostortholoc}
	\left| \int_{\Sph^d} r_j \, \mathrm d\omega \right| \lesssim \| r_j \|_2^2 + \|r_j\|_\frac{4d}{d-4}^\frac{4d}{d-4} \qquad \text{and} \qquad 
	\left| \int_{\Sph^d} \omega_i \, r_j \, \mathrm d\omega \right| \lesssim \| r_j \|_{W^{1,2}}^2 \,,\ i=1,\ldots,d+1 \,.
\end{equation}
Our goal will be to bound $F_2[u_j] - S_d^{(2)}$ from below in terms of $\|r_j\|_{W^{1,2}}^2 + \|r_j\|_{W^{1,4}}^4$.

In order to state our initial remainder bound, we introduce the functionals
\begin{align*}
	E_2[u] & \coloneqq \int_{\Sph^d} e_2(u)\,\mathrm d\omega \,,\\
	\mathcal I_2[r]&\coloneqq \frac{4(d-1)}{d-4}\int_{\mathbb S^d}\left(|\nabla r|^2
	- d \, r^2\right)\mathrm d \omega\,, \\ 
	\mathcal I_3[r]&\coloneqq	\int_{\mathbb S^d}  \left(\frac{4(d-2)}{d-4}r|\nabla r|^2+ \frac{d(d-1)}{2} \, r^3\right) \mathrm d\omega\,,\\	
	\mathcal I_4[r]&\coloneqq \int_{\mathbb S^d}  \left(\frac{1}{2}\left(\frac{4}{d-4}\right)^3|\nabla r|^4+\frac{2(d-2)}{d-4} \, r^2|\nabla r|^2+ \frac{d(d-1)}8\, r^4\right) \mathrm d\omega\,.
\end{align*}	
In terms of these functions the following bound is valid.
	
\begin{lemma}\label{lem:locinitial}
	We have, as $j\to\infty$,
	\begin{align*}
		E_2[u_j] - S_d^{(2)} \|u_j\|_{\frac{4d}{d-4}}^4
		& \geq \mathcal I_2[r_j] + \mathcal I_3[r_j] + \mathcal I_4[r_j]  
		+ \left( \frac{4}{d-4} \right)^3 \int_{\Sph^d} \left( \sigma_1(1+r_j) - \sigma_1(1) \right) |\nabla r_j|^2 \,\mathrm d\omega \\
		& \quad + o(\| r_j \|_{W^{1,2}}^2 + \| r_j \|_{W^{1,4}}^4) \,.
	\end{align*}
\end{lemma}	

Note that the functionals $\mathcal I_k$ are homogeneous of degree $k$. The integral term in the above bound is not homogeneous, but since the difference $\sigma_1(1+r_j) - \sigma_1(1)$ vanishes linearly with $r_j$, its integrand vanishes cubically with $r_j$.

\begin{proof}
	By conformal invariance, we have
	$$
	E_2[u_j] = E_2[(u_j)_{\Psi_j}] = E_2[1+r_j]
	\qquad\text{and}\qquad
	\|u_j\|_{\frac{4d}{d-4}}^4 = \|(u_j)_{\Psi_j} \|_{\frac{4d}{d-4}}^4 = \|1+r_j \|_{\frac{4d}{d-4}}^4\,.
	$$
	Concerning the energy function $E_2$, we expand $e_2(1+r_j)$ and find after simple computations
	\begin{align}\label{eq:bddsigma2}
		E_2[1+r_j] & = \frac{d(d-1)}{8}\, |\Sph^d| + \frac{d(d-1)}{2} \int_{\Sph^d} r_j \,\mathrm d\omega \notag \\
		& \quad + \frac{4(d-1)}{d-4} \int_{\Sph^d} \left( |\nabla r_j|^2 + \frac{3d(d-4)}{16}\, r_j^2 \right) \mathrm d\omega \notag \\
		& \quad + \mathcal I_3[r_j] + \mathcal I_4[r_j] \notag \\
		& \quad + \left( \frac{4}{d-4} \right)^3 \int_{\Sph^d} \left( \sigma_1(1+r_j) - \sigma_1(1) \right) |\nabla r_j|^2 \,\mathrm d\omega \,.
	\end{align}
	Thus, our task is to find an upper bound on $\|1+r_j \|_{\frac{4d}{d-4}}^4$. For this purpose, we use the elementary fact that for every $\kappa>0$ there is a constant $C_\kappa$ such that for all $\rho\in\R$ one has
	\begin{equation}
		\label{eq:elementary}
		|1+\rho|^\frac{4d}{d-4} \leq 1+ \frac{4d}{d-4} \rho + \frac{d}{d-4} \left(  \frac{2(3d+4)}{d-4} +\kappa \right) \rho^2 + \frac{d}{d-4} C_\kappa |\rho|^\frac{4d}{d-4} \,;
	\end{equation}
	see, e.g., \cite[Lemma 3.2]{Figalli2015}. In fact, $C_\kappa = C(1+\kappa^{-\frac{d+12}{d-4}})$ can be chosen with a constant $C$ depending only on $d$. To see this, we note that there are constants $C'$ and $C''$ such that
	$$
	|1+\rho|^\frac{4d}{d-4} \leq 1+ \frac{4d}{d-4} \rho + \frac{d}{d-4} \frac{2(3d+4)}{d-4} \rho^2 + C' |\rho|^3 + C'' |\rho|^\frac{4d}{d-4} \,;
	$$
	The claimed bound \eqref{eq:elementary} now follows from
	$$
	\frac{d-4}{d} C' |\rho|^3 \leq \kappa \rho^2 + C''' \kappa^{-\frac{d+12}{d-4}} |\rho|^\frac{4d}{d-4}
	$$
	for some constant $C'''$.
	
	We take $\rho = r_j(\omega)$ in \eqref{eq:elementary} and integrate the resulting inequality with respect to $\omega$. By concavity of $t\mapsto t^{(d-4)/d}$, we obtain
	\begin{equation}\label{eq:bddsigma0}
		\|1+ r_j\|_{\frac{4d}{d-4}}^4 \leq |\mathbb S^d|^{\frac{d-4}{d}} \left( 1+|\mathbb S^d|^{-1}\int_{\mathbb S^d}\left(4 r_j+\left(\frac{2(3d-4)}{d-4}+\kappa\right) r_j^2+ C_\kappa | r_j|^{\frac{4d}{d-4}}\right)\mathrm d\omega\right).
	\end{equation}
	
	Combining \eqref{eq:bddsigma2} and \eqref{eq:bddsigma0} and recalling that $S_d^{(2)} = (d(d-1)/8) |\Sph^d|^{4/d}$, we obtain
	\begin{align*}
		E_2[1+r_j] - S_d^{(2)} \|1+r_j\|_\frac{4d}{d-4}^4
		& \geq \mathcal I_2[r_j] + \mathcal I_3[r_j] + \mathcal I_4[r_j] \\
		& \quad + \left( \frac{4}{d-4} \right)^3 \int_{\Sph^d} \left( \sigma_1(1+r_j) - \sigma_1(1) \right) |\nabla r_j|^2 \,\mathrm d\omega \\
		& \quad - \frac{d(d-1)}{8} \kappa \int_{\Sph^d} r_j^2\,\mathrm d\omega 
		- \frac{d(d-1)}{8} C_\kappa \int_{\Sph^d} |r_j|^\frac{4d}{d-4} \,\mathrm d\omega \,.
	\end{align*}
	The assertion of the lemma now follows from the fact that
	$$
	\frac{d(d-1)}{8} \kappa \int_{\Sph^d} r_j^2\,\mathrm d\omega 
	+ \frac{d(d-1)}{8} C_\kappa \int_{\Sph^d} |r_j|^\frac{4d}{d-4} \,\mathrm d\omega
	= o( \| r_j \|_2^2 + \| r_j\|_\frac{4d}{d-4}^4 ) \,.
	$$
Indeed, since $4d/(d-4)>4$, it is possible to choose a sequence $(\kappa_j)$ that tends to zero so slowly that $C_{\kappa_j} \|r_j\|_\frac{4d}{d-4}^\frac{4d}{d-4}= o( \| r_j\|_\frac{4d}{d-4}^4 )$.
\end{proof}

To proceed, we will decompose $r_j$ into three different pieces, corresponding to spherical harmonics of certain degrees. We recall that we have an orthogonal decomposition
$$
L^2(\Sph^d) = \bigoplus_{\ell=0}^\infty \mathcal H_\ell \,,
$$
where $\mathcal H_\ell$ is the space of spherical harmonics of degree $\ell$. For more details on spherical harmonics, we refer to \cite[p.~137--152]{Stein1990}. We denote by $\Pi_\ell$ the orthogonal projection in $L^2(\Sph^d)$ onto $\mathcal H_\ell$. Given a parameter $L\geq 1$, we will decompose
$$
r_j = r_j^\lo + r_j^\me + r_j^\hi
$$
with
$$
r_j^\lo := \sum_{\ell=0}^1 \Pi_\ell r_j \,,
\qquad
r_j^\me := \sum_{\ell=2}^L \Pi_\ell r_j \,,
\qquad
r_j^\me := \sum_{\ell=L+1}^\infty \Pi_\ell r_j \,.
$$
We assume throughout that $L$ is independent of $j$. It will be chosen in the proof of Proposition~\ref{prop:loc}. As with all the constants in this paper, its final choice may depend on $d$.

We now state the key technical ingredient in the proof of Proposition~\ref{prop:loc}. It says that for the terms on the right side in Lemma \ref{lem:locinitial} that are not quadratic (that is, all terms except for $\mathcal I_2[r_j]$) only the high frequency component $r_j^\hi$ contributes up to the order that we are interested in.

\begin{lemma}\label{lem:locsimplified}
	For every fixed choice of $L$, we have
	\begin{align}
		\label{eq:locsimplified1a}
		\| r_j \|_{W^{1,2}}^2 & =  \| r_j^\me \|_{W^{1,2}}^2 + \| r_j^\hi \|_{W^{1,2}}^2 +  o(\|r_j\|_{W^{1,2}}^2 + \| r_j \|_{W^{1,4}}^4)   \,,\\
		\label{eq:locsimplified1b}
		\| r_j \|_{W^{1,4}}^4 & =  \| r_j^\hi \|_{W^{1,4}}^4 + o(\|r_j\|_{W^{1,2}}^2 + \| r_j \|_{W^{1,4}}^4) \,.
	\end{align}
	Moreover,	
	\begin{align}
		\label{eq:locsimplified2a}
		\mathcal I_2[r_j] & = \mathcal I_2[r_j^\me] + \mathcal I_2[r_j^\hi] + o(\|r_j\|_{W^{1,2}}^2 + \| r_j \|_{W^{1,4}}^4)\,,\\
		\label{eq:locsimplified2b}
		\mathcal I_3[r_j] & = \mathcal I_3[r_j^\hi] + o( \|r_j\|_{W^{1,2}}^2) \,,\\
		\label{eq:locsimplified2c}
		\mathcal I_4[r_j] & = \mathcal I_4[r_j^\hi] + o( \|r_j\|_{W^{1,2}}^2 + \|r_j\|_{W^{1,4}}^4)\,,
	\end{align}
	and
	\begin{align}
	\notag
		\int_{\Sph^d} &\left( \sigma_1(1+r_j) - \sigma_1(1) \right) |\nabla r_j|^2 \,\mathrm d\omega \\&= \int_{\Sph^d} \left( \sigma_1(1+r_j) - \sigma_1(1) \right) |\nabla r_j^\hi|^2 \,\mathrm d\omega + o( \|r_j\|_{W^{1,2}}^2 + \|r_j\|_{W^{1,4}}^4) \,.	\label{eq:locsimplified3}
	\end{align}
\end{lemma}

We will prove this lemma in the following two subsections. Accepting its conclusion for the moment, we will use it to complete the proof of the local stability result.

\begin{proof}[Proof of Proposition \ref{prop:loc}]
	Let $(u_j)$ be as at the beginning of this section. According to Lemma~\ref{lem:locinitial}, as well as equations \eqref{eq:locsimplified2a}, \eqref{eq:locsimplified2b}, \eqref{eq:locsimplified2c}, and \eqref{eq:locsimplified3} in Lemma \ref{lem:locsimplified}, we have
	\begin{align}\label{eq:locproof}
		E_2[u_j] - S_d^{(2)} \|u_j\|_{\frac{4d}{d-4}}^4
		& \geq \mathcal I_2[r_j^\me] + \mathcal I_2[r_j^\hi] + \mathcal I_3[r_j^\hi] + \mathcal I_4[r_j^\hi] \notag \\
		& \quad - \left( \frac{4}{d-4} \right)^3 \sigma_1(1) \int_{\Sph^d} |\nabla r_j^\hi|^2 \,\mathrm d\omega + o(\| r_j\|_{W^{1,2}}^2 + \|r_j\|_{W^{1,4}}^4) \,.
	\end{align}
	Note that we dropped the term $\int_{\Sph^d} \sigma_1(1+r_j) |\nabla r_j^\hi|^2\,\mathrm d\omega \geq 0$ using the fact that by assumption and by conformal invariance $\sigma_1(1+r_j) = \sigma_1((u_j)_{\Psi_j}) = \sigma_1(u_j)> 0$.
	
	In the following we consider separately the various terms on the right side of \eqref{eq:locproof}. Since the quadratic form $\mathcal I_2$ is positive definite on $(\mathcal H_0 \oplus \mathcal H_1)^\bot$ and since $r_j^\me$ belongs to this space, we have
	$$
	\mathcal I_2[r_j^\me] \gtrsim \| r_j^\me \|_{W^{1,2}}^2 \,.
	$$
	Next, we observe that
	$$
	\mathcal I_2[r_j^\hi] - \left( \frac{4}{d-4} \right)^3 \sigma_1(1) \int_{\Sph^d} |\nabla r_j^\hi|^2 \,\mathrm d\omega \gtrsim \| r_j^\hi \|_{W^{1,2}}^2 \,.
	$$
	Indeed, this is equivalent to the inequality
	$$
	\frac{4(d-1)}{d-4} \left( \ell(\ell+d-1) - d \right) - \frac{2d}{d-4} \, \ell(\ell+d-1) \gtrsim  \ell(\ell+d-1) + 1
	\qquad\text{for all}\ \ell\geq L+1 \,.
	$$
	Using the frequency cut-off, we further bound for later purposes
	$$
	 \| r_j^\hi \|_{W^{1,2}}^2 \geq \frac12 \left( \| r_j^\hi \|_{W^{1,2}}^2 + \left( (L+1)(L+d) + 1 \right)  \| r_j^\hi \|_2^2 \right).
	$$
	
	By dropping the middle term in the definition of $\mathcal I_4$, we find
	$$
	\mathcal I_4[r_j^\hi] \gtrsim \| r_j^\hi \|_{W^{1,4}}^4 \,.
	$$
	Finally, by the Cauchy--Schwarz inequality, we find
	$$
	\mathcal I_3[r_j^\hi] \gtrsim - \| r_j^\hi \|_2 \| r_j^\hi \|_{W^{1,4}}^2 \,.
	$$
	Thus, there is a constant $C$ such that for any $\epsilon>0$ we have
	$$
	\mathcal I_3[r_j^\hi] \geq - \epsilon \| r_j^\hi \|_{W^{1,4}}^4 - C \epsilon^{-1} \| r_j^\hi \|_2^2 \,.
	$$
	Let us show that here the two negative terms on the right side can be controlled by the positive terms coming from the other functionals. We choose $\epsilon>0$ so small (independent of $j$) that the term $- \epsilon \| r_j^\hi \|_{W^{1,4}}^4$ is compensated by half of the term in the lower bound on $\mathcal I_4[r_j^\hi]$. Then we choose $L$ so large that the term $- C \epsilon^{-1} \| r_j^\hi \|_2^2$ is compensated by the term proportional to $((L+1)(L+d)+1) \| r_j^\hi \|_{2}^2$.
	
	Inserting all these bounds into the right side of \eqref{eq:locproof}, we arrive at
	\begin{align*}
		E_2[u_j] - S_d^{(2)} \|u_j\|_{\frac{4d}{d-4}}^4
		& \gtrsim \| r_j^\me \|_{W^{1,2}}^2 + \| r_j^\hi \|_{W^{1,2}}^2 + \|r_j^\hi \|_{W^{1,4}}^4
		+ o(\| r_j\|_{W^{1,2}}^2 + \|r_j\|_{W^{1,4}}^4) \,.
	\end{align*}
	Using \eqref{eq:locsimplified1a} and \eqref{eq:locsimplified1b} in Lemma \ref{lem:locsimplified}, we see that the right side is equal to $\| r_j \|_{W^{1,2}}^2 + \|r_j \|_{W^{1,4}}^4 + o(\| r_j\|_{W^{1,2}}^2 + \|r_j\|_{W^{1,4}}^4)$.
	
	Since
\begin{multline*}
		\inf_\Psi \left( \| (u_j)_\Psi - 1 \|_{W^{1,2}}^2 + \| (u_j)_\Psi - 1 \|_{W^{1,4}}^4 \right) \\\leq \| (u_j)_{\Psi_j} - 1 \|_{W^{1,2}}^2 + \| (u_j)_{\Psi_j} - 1 \|_{W^{1,4}}^4 = \| r_j\|_{W^{1,2}}^2 + \|r_j\|_{W^{1,4}}^4
	\end{multline*}
	and since $\| u_j \|_{\frac{4d}{d-4}}^\frac{4d}{d-4}=|\Sph^d|$, we have shown that
	$$
	\liminf_{j\to\infty} \frac{F_2[u_j] - S_d^{(2)}}{\inf_\Psi \left( \| (u_j)_\Psi - 1 \|_{W^{1,2}}^2 + \| (u_j)_\Psi - 1 \|_{W^{1,4}}^4 \right)} \gtrsim 1
	$$
	with an implicit constant depending only on $d$. This completes the proof of Proposition \ref{prop:loc}.
\end{proof}

It remains to prove Lemma \ref{lem:locsimplified}. This will be accomplished in the following two subsections.

	
\subsection{Controlling the low frequency part}

Our goal in this subsection is to prove the `quadratic' assertions in Lemma \ref{lem:locsimplified}, that is, \eqref{eq:locsimplified1a} and \eqref{eq:locsimplified2a}. In both cases the assertion reduces to the fact that $r_j^\lo$ is `essentially zero'. This will be a consequence of the almost-orthogonality conditions \eqref{eq:almostortholoc}.

\begin{proof}[Proof of Lemma \ref{lem:locsimplified}. Part 1]
	\emph{Step 1.} We shall show that
	\begin{equation*}
		\| r_j^\lo \|_{W^{1,2}} \lesssim \| r_j \|_{W^{1,2}}^2 + \|r_j\|_{W^{1,4}}^4 \,.
	\end{equation*}
	To prove this, we treat the contributions $\Pi_0 r_j$ and $\Pi_1 r_j$ to $r_j^\lo$ separately.
	
	Noting that $\mathcal H_0$ is the space of constant functions, we see that $\Pi_0 r_j$ is the constant with value
	$$
	\Pi_0 r_j = |\Sph^d|^{-1} \int_{\Sph^d} r_j \,\mathrm d\omega \,.
	$$
	By the almost-orthogonality condition \eqref{eq:almostortholoc} and the fact that $\|r_j\|_{W^{1,4}}\to 0$, we have
	$$
	\| \Pi_0 r_j \|_{W^{1,2}} = \| \Pi_0 r_j \|_2 \lesssim \left| \int_{\Sph^d} r_j \,\mathrm d\omega \right| \lesssim \| r_j \|_2^2 + \|r_j\|_\frac{4d}{d-4}^\frac{4d}{d-4} \lesssim \| r_j \|_{W^{1,2}}^2 + \|r_j\|_{W^{1,4}}^4 \,.
	$$
	
	Similarly, $\mathcal H_1$ is the space of linear functions, and we see that
	$$
	(\Pi_1 r_j)(\omega) = \sum_{i=1}^{d+1} \omega_i \ \frac{d+1}{|\Sph^d|} \int_{\Sph^d} \omega'_i \, r_j(\omega')\,\mathrm d\omega' \,.
	$$
	By the second group of almost-orthogonality conditions \eqref{eq:almostortholoc}, we have
	$$
	\| \Pi_1 r_j \|_{W^{1,2}} \lesssim \|\Pi_1 r_j\|_{2} \lesssim \left( \sum_i \left| \int_{\Sph^d} \omega'_i \, r_j(\omega')\,\mathrm d\omega' \right|^2 \right)^{1/2}
	\lesssim \|r_j \|_{W^{1,2}}^2 \,.
	$$
	
	Since $\| r_j^\lo\|_{W^{1,2}}^2 = \| \Pi_0 r_j \|_{W^{1,2}}^2 + \|\Pi_1 r_j\|_{W^{1,2}}^2$, we obtain the claimed bound.
	
	\medskip
	
	\emph{Step 2.} Using the bound from Step 1, it is easy to complete the proof. By orthogonality of spherical harmonics in $L^2(\Sph^d)$ and $W^{1,2}(\Sph^d)$, we have
	\begin{align*}
		\| r_j \|_{W^{1,2}}^2 & =  \| r_j^\lo \|_{W^{1,2}}^2 + \| r_j^\me \|_{W^{1,2}}^2 + \| r_j^\hi \|_{W^{1,2}}^2 \,,\\
		\mathcal I_2[r_j] & = \mathcal I_2[r_j^\lo] + \mathcal I_2[r_j^\me] + \mathcal I_2[r_j^\hi] \,.
	\end{align*}
	The assertion now follows from $|\mathcal I_2[r_j^\lo]|\lesssim \| r_j^\lo\|_{W^{1,2}}$ in view of the bound from Step 1.
\end{proof}


\subsection{Controlling the medium frequency part}

In the previous subsection we have proved \eqref{eq:locsimplified1a} and \eqref{eq:locsimplified2a} in Lemma \ref{lem:locsimplified}. Our goal in the present subsection is to prove the remaining equations in this lemma. Those concern terms that vanish faster than quadratically and the assertion is that only the contribution $r_j^\hi$ is relevant for them. This comes basically from the fact that in the space of spherical harmonics of degree $\leq L$, which is finite-dimensional, the (stronger) norm $W^{1,4}(\Sph^d)$ is equivalent to the (weaker) norm $W^{1,2}(\Sph^d)$.

\begin{proof}[Proof of Lemma \ref{lem:locsimplified}. Part 2]
	In the following proof the distinction between $r_j^\lo$ and $r_j^\me$ is irrelevant, and it is convenient to set
	$$
	r_j^< := r_j^\lo + r_j^\me \,.
	$$
	These functions belong to the space $\bigoplus_{\ell=0}^L \mathcal H_\ell$, so they are smooth as (restrictions to $\Sph^d$ of) polynomials of degree $\leq L$. Moreover, since in this finite-dimensional space all norms are equivalent, for any $k\in\N_0$ there is a constant such that for all $j$ we have
	\begin{equation}
		\label{eq:locsimplifiedproof1}
		\| r_j^< \|_{C^k} \lesssim_k \|r_j \|_{W^{1,2}} \,.
	\end{equation}	
	The constant in this inequality depends on $L$, but it is considered fixed in this proof, and its dependence will not be tracked. In what follows, we will use inequality \eqref{eq:locsimplifiedproof1} only with $k=1$ and $k=2$.
			
	\medskip
	
	\emph{Proof of \eqref{eq:locsimplified1b}.} By expanding the quartic power and using some elementary estimates, we find pointwise on $\Sph^d$,
	\begin{align*}
		\left| |\nabla r_j^\hi|^4 - |\nabla r_j|^4 \right| & \lesssim \sum_{j=1}^4 |\nabla r_j^<|^j |\nabla r_j^\hi|^{4-j} \lesssim \sum_{j=1}^4 |\nabla r_j^<|^j |\nabla r_j|^{4-j}
		\lesssim |\nabla r_j^<| |\nabla r_j|^3 + |\nabla r_j^<|^4 \,.
	\end{align*}
	By H\"older's inequality, this and a similar bound for $| (r_j^\hi)^4 - r_j^4|$ imply that
	$$
	\left| \| r_h^\hi \|_{W^{1,4}}^4 - \| r_j \|_{W^{1,4}}^4 \right| \lesssim \| r_j^< \|_{W^{1,4}}  \| r_j \|_{W^{1,4}}^3 + \| r_j^< \|_{W^{1,4}}^4 \,.
	$$
	According to \eqref{eq:locsimplifiedproof1} with $k=1$ (we only use the weaker inequality with a $W^{1,4}$-norm on the left), this term is 
	$$
	\lesssim \| r_j \|_{W^{1,2}}  \| r_j \|_{W^{1,4}}^3 + \| r_j \|_{W^{1,2}}^4 \,.
	$$
	Here the second term is $o(\| r_j \|_{W^{1,2}}^2)$ and therefore acceptable. For the first term, we bound
	$$
	\| r_j \|_{W^{1,2}}  \| r_j \|_{W^{1,4}}^3 \lesssim \epsilon_j \| r_j \|_{W^{1,2}}^2 + \epsilon_j^{-1} \| r_j \|_{W^{1,4}}^6 \,.
	$$
	Choosing $\epsilon_j = \|r_j\|_{W^{1,4}}$, the right side is $o(\| r_j \|_{W^{1,2}}^2+\|r_j\|_{W^{1,4}}^4)$ as well. This completes the proof of \eqref{eq:locsimplified1b}.
	
	\medskip
	
	\emph{Proof of \eqref{eq:locsimplified2c}.} The terms involving $|\nabla r_j|^4$ and $r_j^4$ in $\mathcal I_4[r_j]$ are handled in the exact same way as in the proof of \eqref{eq:locsimplified1b}. The argument for the remaining term involving $r_j^2|\nabla r_j|^2$ follows along the same lines. We omit the details.
	
	\medskip
	
	\emph{Proof of \eqref{eq:locsimplified2b}.} By expanding the cubic power and using some elementary estimates, we find pointwise on $\Sph^d$,
	\begin{align*}
		\left| r_j^\hi |\nabla r_j^\hi|^2 - r_j |\nabla r_j|^2 \right| & \lesssim |r_j^<| |\nabla r_j|^2 + |r_j^\hi| \left( |\nabla r_j^<|^2 + |\nabla r_j^<||\nabla r_j^\hi| \right) \\
		& \lesssim |r_j^<| \left( |\nabla r_j|^2 + |\nabla r_j^<|^2 \right) + |r_j| \left( |\nabla r_j^<|^2 + |\nabla r_j^<||\nabla r_j| \right).
	\end{align*}
	In view of \eqref{eq:locsimplifiedproof1} with $k=1$, we can write
	$$
	\left| r_j^\hi |\nabla r_j^\hi|^2 - r_j |\nabla r_j|^2 \right|  \lesssim
	\| r_j \|_{W^{1,2}} \left( |\nabla r_j|^2 + \| r_j \|_{W^{1,2}}^2 \right) + |r_j| \left( \| r_j \|_{W^{1,2}}^2 + \| r_j \|_{W^{1,2}} |\nabla r_j| \right).
	$$
	Integrating this and a similar bound for $|(r_j^\hi)^3 - r_j^3|$ over $\Sph^d$ leads to
	\begin{align*}
		\left| \mathcal I_3[r_j^\hi] - \mathcal I_3[r_j] \right|
		& \lesssim \| r_j \|_{W^{1,2}}^3 \,.
	\end{align*}
	Clearly, this is $o(\| r_j \|_{W^{1,2}}^2)$, proving \eqref{eq:locsimplified2b}.
	
	\medskip
	
	\emph{Proof of \eqref{eq:locsimplified3}.} Writing
	$$
	\sigma_1(1+r_j) - \sigma_1(1) = -\frac{d-4}8 \Delta\left(2r_j + r_j^2\right) - |\nabla r_j|^2 + \frac d2 \left( \frac{d-4}{4} \right)^2 \left( 2r_j + r_j^2 \right)
	$$
	and
	$$
	|\nabla r_j|^2 - |\nabla r_j^\hi|^2 = - |\nabla r_j^<|^2 + 2 \nabla r_j^<\cdot \nabla r_j \,,
	$$
	we obtain
	$$
	\int_{\Sph^d} \left( \sigma_1(1+r_j) - \sigma_1(1) \right) \left( |\nabla r_j|^2 - |\nabla r_j^\hi|^2 \right)\mathrm d \omega = A+B
	$$
	with
	\begin{align*}
		A & := - \frac{d-4}8 \int_{\Sph^d} \left( \Delta(2r_j + r_j^2) \right) \left( - |\nabla r_j^<|^2 + 2 \nabla r_j^<\cdot \nabla r_j \right)\mathrm d\omega \,, \\
		B & := \int_{\Sph^d} \left( - |\nabla r_j|^2 + \frac d2 \left( \frac{d-4}{4} \right)^2 \left( 2r_j + r_j^2 \right) \right) \left( - |\nabla r_j^<|^2 + 2 \nabla r_j^<\cdot \nabla r_j \right)\mathrm d \omega \,. 
	\end{align*}
	
	When we multiply out the terms in the integrand of $B$, we see that every term contains at least one factor of $r_j^<$ or $\nabla r_j^<$. Therefore, we can treat this term in the same way as in the proof of \eqref{eq:locsimplified1b} and \eqref{eq:locsimplified2c} (for the terms homogeneous of degree four) and of \eqref{eq:locsimplified2b} (for the terms homogeneous of degree three). We deduce that $B = o(\|r_j\|^2_{W^{1,2}} + \|r_j\|_{W^{1,4}}^4)$.
	
	The argument for $A$ is more involved. We begin by integrating by parts to write
	$$
	A = A_1 + A_2
	$$
	with
	\begin{align*}
		A_1 & := - \frac{d-4}4 \int_{\Sph^d} (1+r_j) \nabla r_j \cdot \nabla \left( |\nabla r_j^<|^2 \right)\mathrm d\omega \,, \\
		A_2 & := \frac{d-4}2 \int_{\Sph^d} (1+r_j) \nabla r_j \cdot \nabla \left( \nabla r_j^<\cdot \nabla r_j \right)\mathrm d \omega \,.
	\end{align*}
	From \eqref{eq:locsimplifiedproof1} with $k=2$, we know that $\nabla \left(|\nabla r_j^<|^2\right)$ is bounded pointwise by $\| r_j \|_{W^{1,2}}^2$. Thus,
	$$
	|A_1| \lesssim ( 1 + \|r_j\|_2) \|r_j\|_{W^{1,2}}^3 = o( \|r_j\|_{W^{1,2}}^2) \,.
	$$
	To deal with $A_2$, we apply Lemma \ref{lem:productrule} with $X=\nabla r_j$ and $Y=\nabla r_j^<$. We deduce that
	$$
	\left| \nabla r_j \cdot \nabla \left( \nabla r_j^< \cdot \nabla r_j \right) - \frac12 \nabla r_j^< \cdot\nabla \left(|\nabla r_j|^2\right) \right| \lesssim |\nabla r_j|^2 \left( |\nabla^2 r_j^<| + |\nabla r_j^<| \right).
	$$
	We multiply this inequality by $(1+r_j)$ and integrate. Using $|\nabla^2 r_j^<| + |\nabla r_j^<| \lesssim \|r_j\|_{W^{1,2}}$ by \eqref{eq:locsimplifiedproof1} with $k=2$ and applying H\"older's inequality, we find
	\begin{align*}
		\left| A_2  - \frac{d-4}{4} \int_{\Sph^d} (1+r_j) \nabla r_j^< \cdot \nabla \left(|\nabla r_j|^2\right)\, \mathrm d\omega \right| \lesssim \| r_j \|_{W^{1,2}} \left( \| \nabla r_j\|_2^2 + \|r_j\|_3 \|\nabla r_j\|_3^2 \right).
	\end{align*}
	By H\"older's inequality $\| f \|_3 \leq \|f\|_2^{1/3} \|f\|_4^{2/3}$ for $f\in L^4(\Sph^d)$, so the right side is
	\begin{align*}
		& \lesssim \| r_j \|_{W^{1,2}} \left( \| \nabla r_j\|_2^2 + \|r_j\|_2^{1/3}\|r_j\|_4^{2/3} \|\nabla r_j\|_2^{2/3} \|\nabla r_j\|_4^{4/3} \right) \\
		& \leq \| r_j \|_{W^{1,2}} \left( \| r_j\|_{W^{1,2}}^2 + \|r_j\|_{W^{1,2}} \| r_j\|_{W^{1,4}}^2 \right).
	\end{align*}
	This is $o(\|r_j\|_{W^{1,2}}^2 + \|r_j\|_{W^{1,4}}^4)$, as desired.
	
 Finally, in the remaining integral term, we integrate by parts, which gives 
	$$
	\int_{\Sph^d} (1+r_j) \nabla r_j^< \cdot \nabla \left(|\nabla r_j|^2\right)\, \mathrm d\omega
	= - \int_{\Sph^d} (1+r_j) (\Delta r_j^<) |\nabla r_j|^2 \,\mathrm d\omega - \int_{\Sph^d} \nabla r_j \cdot \nabla r_j^< |\nabla r_j|^2\,\mathrm d\omega \,.
	$$
	Recalling $|\Delta r_j^<| + |\nabla r_j^<|\lesssim \|r_j\|_{W^{1,2}}$ by \eqref{eq:locsimplifiedproof1} with $k=2$, we see that the right side is bounded by
	$$
	\|r_j \|_{W^{1,2}} \left( \|\nabla r_j\|_2^2 + \|r_j\|_3 \|\nabla r_j\|_3^2 + \|\nabla r_j\|_3^3 \right).
	$$
	Dealing with the $L^3$-norms as before, we see that this is $o(\|r_j\|_{W^{1,2}}^2 + \|r_j\|_{W^{1,4}}^4)$. Thus, we have shown that $A=o(\|r_j\|_{W^{1,2}}^2 + \|r_j\|_{W^{1,4}}^4)$, which concludes the proof of \eqref{eq:locsimplified3}.	
\end{proof}


\section{Sharpness of the results}\label{sec:4} 

In this section we prove that our main result, Theorem \ref{thm}, is optimal in two different respects. First, the exponent 2 of the $W^{1,2}$-norm cannot be decreased. Second, the exponent 4 of the $W^{1,4}$-norm cannot be decreased. Both items are discussed in the next two subsections, respectively.


\subsection{Sharpness of the quadratic $W^{1,2}(\mathbb S^d)$-growth}

In this subsection, we exhibit a family $(u_\epsilon)\subset C^\infty(\Sph^d)$ such that $u_\epsilon>0$ and $\sigma_1(u_\epsilon)>0$ for all sufficiently small $\epsilon>0$, as well as $\| u_\epsilon\|_{\frac{4d}{d-4}} = \| 1 \|_{\frac{4d}{d-4}}$ and 
\begin{equation}
	\label{eq:sharppow2}
	\lim_{\epsilon\to 0} \inf_{\Psi}\|(u_\epsilon)_\Psi -1\|_{W^{1,2}}^2 = 0
	\qquad\text{and}\qquad
	\limsup_{\epsilon\to 0} \frac{F_2[u_\epsilon]- S_d^{(2)}}{\inf_{\Psi}\|(u_\epsilon)_\Psi -1\|_{W^{1,2}}^2} < \infty \,.
\end{equation}
Note that here we do not minimize with respect to $\lambda$, but instead we impose the normalization condition $\| u_\epsilon\|_{\frac{4d}{d-4}} = \| 1 \|_{\frac{4d}{d-4}}$. This corresponds to what was actually shown in the proof of Theorem~\ref{thm}.

To prove \eqref{eq:sharppow2}, we fix a function $0\not\equiv\phi\in C^\infty(\Sph^d)$ with $\int_{\Sph^d} \phi \,\mathrm d\omega = 0$ and $\int_{\Sph^d} \omega\, \phi \,\mathrm d\omega=0$ and consider the family 
$$
u_\epsilon \coloneqq \lambda_\epsilon \left( 1+\epsilon \phi \right)
\qquad\text{with}\qquad
\lambda_\epsilon:= \frac{\| 1 \|_{\frac{4d}{d-4}}}{\|1+\epsilon\phi\|_{\frac{4d}{d-4}}} \,.
$$
As $\phi$ and its derivatives are bounded, we can easily verify that $u_\epsilon>0$ and $\sigma_1(u_\epsilon)>0$ for all sufficiently small $\epsilon>0$. Arguing as in the proof of Lemma \ref{lem:locinitial}, we find
$$
F_2[u_\epsilon] - S_d^{(2)} = \| 1 + \epsilon \phi \|_\frac{4d}{d-4}^{-4} \left( E_2[1+\epsilon\phi] - S_d^{(2)} \| 1 + \epsilon \phi \|_\frac{4d}{d-4}^{4} \right) 
= \epsilon^2 |\Sph^d|^{-\frac{d-4}{d}} \mathcal I_2[\phi] + o(\epsilon^2) \,.
$$

The proof of \eqref{eq:sharppow2} will therefore be complete once we have shown
\begin{equation}\label{eq:distH1}
	\inf_{\Psi}\|(u_\epsilon)_\Psi -1\|_{W^{1,2}}^2 = \epsilon^2 \|\phi\|^2_{W^{1,2}}+o(\epsilon^2)\,.
\end{equation}
Since
$$
\|(u_\epsilon)_\Psi -1\|_{W^{1,2}} = \lambda_\epsilon \left\|(1+\epsilon\phi)_\Psi - \lambda_\epsilon^{-1} \right\|_{W^{1,2}}
$$
and $\lambda_\epsilon^{-1} = 1 + \mathcal O(\epsilon^2)$ in view of the assumption $\int_{\mathbb S^d} \phi\,\mathrm d\omega = 0$, the claim \eqref{eq:distH1} will follow if we can prove
\begin{equation}\label{eq:distH1alt}
	\inf_{\Psi}\|(1+\epsilon\phi)_\Psi -1\|_{W^{1,2}}^2 = \epsilon^2 \|\phi\|^2_{W^{1,2}}+o(\epsilon^2)\,.
\end{equation}
By choosing $\Psi$ as the identity, we see that the left side here is at most $\epsilon^2 \|\phi\|^2_{W^{1,2}}$. On the other hand, thanks to Lemma \ref{lem:att}, we know that for all sufficiently small $\epsilon>0$ there are $\Psi_\epsilon$ such that 
\begin{align}\notag
	\inf_{\Psi}\|(1+\epsilon\phi)_\Psi -1\|_{W^{1,2}}^2 & = \|(1+\epsilon\phi)_{\Psi_\epsilon} -1\|_{W^{1,2}}^2 \\
	&=\|(1)_{\Psi_\epsilon} - 1\|_{W^{1,2}}^2 + 2\epsilon \langle (1)_{\Psi_\epsilon}-1, (\phi)_{\Psi_\epsilon}\rangle_{W^{1,2}} + \epsilon^2 \|(\phi)_{\Psi_\epsilon}\|_{W^{1,2}}^2\,.\label{eq:distH1asympt}
\end{align}
We can bound the left side from above by $\epsilon^2 \|\phi\|_{W^{1,2}}^2$ and the mixed term on the right from below by
$$
2\epsilon \langle (1)_{\Psi_\epsilon}-1, (\phi)_{\Psi_\epsilon}\rangle_{W^{1,2}}
\geq - \frac12 \|(1)_{\Psi_\epsilon} - 1\|_{W^{1,2}}^2 - 2 \epsilon^2 \|(\phi)_{\Psi_\epsilon}\|_{W^{1,2}}^2 \,.
$$
Bounding $\|(\phi)_{\Psi_\epsilon}\|_{W^{1,2}}\lesssim \|(\phi)_{\Psi_\epsilon}\|_{W^{1,4}} \lesssim \|\phi \|_{W^{1,4}}$ by Lemma \ref{lem:confbdd}, we deduce from the upper bound in \eqref{eq:distH1alt} that $\|1-(1)_{\Psi_\epsilon}\|_{W^{1,2}}\lesssim \epsilon$. Meanwhile, arguing as in the proof of Lemma \ref{lem:almostortho}, we see that $\|1-(1)_{\Psi_\epsilon}\|_{W^{1,2}}^2 = \operatorname{const} |\xi_\epsilon|^2 + o(|\xi_\epsilon|^2)$, so we infer that $|\xi_\epsilon|\lesssim \epsilon$. Arguing again as in Lemma~\ref{lem:almostortho}, we find $\| (\phi)_{\Psi_\epsilon} - \phi \|_{W^{1,2}} \lesssim |\xi_\epsilon| \lesssim \epsilon$. Inserting this into  \eqref{eq:distH1asympt}, we obtain
$$
\inf_{\Psi}\|(1+\epsilon\phi)_\Psi -1\|_{W^{1,2}}^2
= \|(1)_{\Psi_\epsilon} - 1\|_{W^{1,2}}^2 + 2\epsilon \langle (1)_{\Psi_\epsilon}-1, \phi \rangle_{W^{1,2}} + \epsilon^2 \|\phi \|_{W^{1,2}}^2 + o(\epsilon^2) \,.
$$
Looking once again at the proof of Lemma \ref{lem:almostortho}, we see that
$$
\langle (1)_{\Psi_\epsilon}-1, \phi \rangle_{W^{1,2}} = \operatorname{const} \int_{\Sph^d} \xi_\epsilon\cdot \omega\, \phi \,\mathrm d\omega + o(|\xi_\epsilon|) = o(\epsilon) \,.
$$
Here we used the orthogonality condition $\int_{\Sph^d} \omega\, \phi \,\mathrm d\omega = 0$. Thus,
$$
\inf_{\Psi}\|(1+\epsilon\phi)_\Psi -1\|_{W^{1,2}}^2
= \|(1)_{\Psi_\epsilon} - 1\|_{W^{1,2}}^2 + \epsilon^2 \|\phi \|_{W^{1,2}}^2 + o(\epsilon^2) \,.
$$
Dropping the nonnegative first term, we arrive at the claimed identity \eqref{eq:distH1alt}. 


\subsection{Sharpness of the quartic $W^{1,4}(\mathbb S^d)$-growth}

In this subsection we exhibit a sequence $(u_j)\subset C^\infty(\Sph^d)$ such that $u_j>0$ and $\sigma_1(u_j)>0$ for all $j$, as well as $\| u_j\|_{\frac{4d}{d-4}} = \| 1 \|_{\frac{4d}{d-4}}$ and 
\begin{equation}
	\label{eq:sharppow4}
	\lim_{j\to\infty} \inf_{\Psi}\|(u_j)_\Psi -1\|_{W^{1,4}}^4 = 0
	\qquad\text{and}\qquad
	\limsup_{j\to \infty} \frac{F_2[u_j]- S_d^{(2)}}{\inf_{\Psi}\|(u_j)_\Psi -1\|_{W^{1,4}}^4} < \infty \,.
\end{equation}
As in \eqref{eq:sharppow2}, instead of minimizing with respect to $\lambda$, we impose the normalization condition $\| u_\epsilon\|_{\frac{4d}{d-4}} = \| 1 \|_{\frac{4d}{d-4}}$.

The elements in the sequence $(u_j)$ in \eqref{eq:sharppow4} consist of two bubbles with eventually disjoint supports on different scales. A similar approach for the sharpness of the $p$-Sobolev inequality on $\mathbb R^d$ can be found in \cite{Figalli2022}.

It is convenient to take $(u_j)$ as a subsequence of the two-parameter family of functions
$$
u_{\delta,\xi} := \lambda_{\delta,\xi} \left(1+ \delta (1)_{\Psi_\xi} \right)
\qquad\text{with}\qquad
\lambda_{\delta,\xi} := \frac{\|1\|_{\frac{4d}{d-4}}}{\|1 + \delta (1)_{\Psi_\xi}\|_{\frac{4d}{d-4}}} \,,
$$
corresponding to a suitably chosen sequence $(\delta_j,\xi_j)$ with $\delta_j\to 0$ and $|\xi_j|\to 1$. We explain this in Step 4 below, after having established all the necessary facts in the previous three steps.

\medskip

\textit{Step 1.} We show that $\sigma_1(u_{\delta,\xi})>0$ if $|\xi|$ is sufficiently close to $1$, uniformly in $\delta>0$. 

Since $u\mapsto \sigma_1(u)$ is homogeneous, it suffices to show the corresponding fact for $\sigma_1(1+\delta (1)_{\Psi_\xi})$. We have
$$
\sigma_1(1+\delta (1)_{\Psi}) = \sigma_1(1) + \delta^2 \sigma_1((1)_{\Psi}) - \delta \frac{d-4}{4} \Delta (1)_{\Psi} + \delta d \left( \frac{d-4}4 \right)^2 (1)_{\Psi} \,.
$$
Clearly, $\sigma_1(1) = (d/2) ((d-4)/4)^2$ and, by conformality of $\Psi$ and $\sigma_1^{\Psi^*g_*}=\sigma_1^{g_*} = d/2$,
$$
\sigma_1((1)_{\Psi}) = \left(\frac{d-4}{4} \right)^2 (1)_\Psi^{\frac{2d}{d-4}} \sigma_1^{(1)_\Psi^\frac{8}{d-4} g_*} = \left(\frac{d-4}{4} \right)^2 (1)_\Psi^{\frac{2d}{d-4}} \sigma_1^{\Psi^* g_*} = \frac d2\, \left(\frac{d-4}{4} \right)^2 (1)_\Psi^{\frac{2d}{d-4}} \,.
$$
Using the explicit form of $J_{\Psi_\xi}$ in \eqref{eq:psijac}, we compute
$$
\Delta (1)_{\Psi_\xi} = \frac{d(d-4)}4 \, (1)_{\Psi_\xi} \, \frac{|\xi|^2 - (\xi\cdot\omega)^2 - 2(1-2\xi\cdot\omega+|\xi|^2)\xi\cdot\omega}{|\omega-\xi|^4}
$$
and write
$$
-\frac{d-4}{4} \Delta (1)_{\Psi} + d \left( \frac{d-4}4 \right)^2 (1)_{\Psi}
= \frac d2\, \left(\frac{d-4}{4} \right)^2 \, (1)_{\Psi_\xi} \, \frac{Q_{|\xi|}(\xi\cdot\omega)}{|\omega-\xi|^4}
$$
with
$$
Q_{|\xi|}(t) := 2(-|\xi|^2 + t^2 + 2t(1-2t + |\xi|^2) + (1-2t + |\xi|^2)^2) \,.
$$
Thus, we arrive at
$$
\sigma_1(1+\delta (1)_{\Psi_\xi}) = \frac d2\, \left(\frac{d-4}{4} \right)^2 \left( 1 + \delta^2 \, (1)_{\Psi_\xi}^{\frac{2d}{d-4}} + \delta (1)_{\Psi_\xi} \frac{Q_{|\xi|}(\xi\cdot\omega)}{|\omega-\xi|^4} \right).
$$
It turns out that the right side is positive -- not only for $|\xi|$ close to $1$, but for any $|\xi|\in [0,1)$. Clearly, the terms of order $\delta^0$ and $\delta^2$ are positive. Next, we study the term of order $\delta$. Note that $Q_{|\xi|}$ is a polynomial of degree two with positive leading order coefficient. Moreover, $Q_{|\xi|}(|\xi|)>0$ and $Q_{|\xi|}'(|\xi|)\leq 0$. Thus, we have $Q_{|\xi|}(t)>0$ for all $t\leq |\xi|$ and, since  $|\xi\cdot\omega|\leq|\xi|$, also $Q_{|\xi|}(\xi\cdot\omega)> 0$ for all $\omega\in\Sph^d$. This proves the claimed positivity.

\medskip

\textit{Step 2.} We show that as $|\xi|\to 1$ we have, uniformly in $0<\delta\leq 1$, 
\begin{equation}
	\label{eq:sharppow4energy}
	E_2[1+\delta(1)_{\Psi_\xi}] = (1 + \delta^4) \, E_2[1] + o_{|\xi|\to 1}(1)
\end{equation}
and 
\begin{equation}
	\label{eq:sharppow4norm}
	\| 1+\delta (1)_{\Psi_{\xi}}\|_{\frac{4d}{d-4}}^\frac{4d}{d-4} = |\mathbb S^d|\left(1+\delta^{\frac{4d}{d-4}} \right)+ o_{|\xi|\to\infty}(1)\,.
\end{equation}

For the proof of the latter expansion, we argue similarly as in the proof of Lemma \ref{lem:mobcomp}. Taking any sequence $(\xi_j)\subset B_1(0)$ with $\xi_j\to\omega$ and $|\omega|=1$ and any bounded sequence  $(\delta_j)\subset (0,1]$, we see that
$$
\delta_j (1)_{\Psi_{\xi_j}} \to 0
\qquad\text{pointwise in}\ \Sph^d \setminus\{\omega\} 
$$ and conclude \eqref{eq:sharppow4norm} using the Brezis--Lieb lemma and the fact that $\| (1)_{\Psi_{\xi_j}}\|_{\frac{4d}{d-4}}^\frac{4d}{d-4}= \| 1\|_{\frac{4d}{d-4}}^\frac{4d}{d-4} = |\Sph^d|$.

To prove \eqref{eq:sharppow4energy}, we expand $e_2(1+\delta_j(1)_{\Psi_j})$ in powers of $\delta_j$, apply Lemma \ref{lem:subcritical} to the terms of order one, two, and three, and use the conformal invariance for the term of order four. In this way we arrive at
\begin{equation}
	E_2[1+\delta_j (1)_{\Psi_{\xi_j}}] = (1+\delta_j^4)\, E_2[1] - \delta_j^3\left(\frac{4}{d-4}\right)^2 \int_{\mathbb S^d}   (\Delta(1)_{\Psi_{\xi_j}}) |\nabla (1)_{\Psi_{\xi_j}}|^2\, \mathrm d\omega + o_{j\to\infty}(1)\,. \label{eq:expE}
\end{equation}
To discard the remaining third order term in \eqref{eq:expE}, we recall $(1)_{\Psi_\xi} = J_{\Psi_\xi}^{(d-4)/(4d)}$ with $J_{\Psi_\xi}$ given in \eqref{eq:psijac}, which yield
$$
|\nabla (1)_{\Psi_\xi}(\omega) | \lesssim \frac{(1-|\xi|^2)^\frac{d-4}4}{|\omega-\xi|^\frac{d-2}2}
\qquad\text{and}\qquad
|\Delta (1)_{\Psi_\xi}(\omega) | \lesssim \frac{(1-|\xi|^2)^\frac{d-4}4}{|\omega-\xi|^\frac{d}2} \,.
$$
Thus,
$$
\left| \int_{\mathbb S^d}   (\Delta(1)_{\Psi_{\xi_j}}) |\nabla (1)_{\Psi_{\xi_j}}|^2\, \mathrm d\omega \right| \lesssim (1-|\xi|^2)^{\frac{3(d-4)}4} \int_{\Sph^d} \frac{\mathrm d\omega}{|\omega-\xi|^{\frac{3d-4}2}} \,.
$$
From \eqref{eq:elementaryintegral} we deduce that the integral in \eqref{eq:expE} is bounded by $$(1-|\xi|^2)^{\frac{3(d-4)}4 - \frac{3d-4}{2} + d} = (1-|\xi|^2)^{\frac d4-1} \to 0\,,$$ which concludes the proof of \eqref{eq:sharppow4energy}.

\medskip

\textit{Step 3.} We show that there is a constant $c>0$ such that for all $\delta\in(0,c]$ and all $\xi$ with $1-|\xi|\in(0,c]$ we can bound
\begin{equation}
	\label{eq:distw14}
	\delta^4 \gtrsim \inf_\Psi \left\| (u_{\delta,\xi})_\Psi - 1 \right\|_{W^{1,4}}^4 \gtrsim \delta^4 \,,
\end{equation} provided that $|\lambda_{\delta,\xi} - 1|\leq c \delta$. Indeed, we shall show that if $\delta\in(0,c]$ and $1-|\xi|\in(0,c]$, then
\begin{equation}
	\label{eq:distw14alt}
	\delta^4 \gtrsim \inf_\Psi \left\| (1+\delta(1)_{\Psi_\xi})_\Psi - 1 \right\|_{W^{1,4}}^4 \gtrsim \delta^4 \,. 
\end{equation}
After possibly decreasing the constant $c$ (depending on the implicit constant in \eqref{eq:distw14alt}), we can argue as in the previous subsection to see that \eqref{eq:distw14} holds under the additional assumption $|\lambda_{\delta,\xi} - 1|\leq c \delta$.

Let us turn to the proof of \eqref{eq:distw14alt}. The upper bound follows immediately by taking $\Psi$ as the identity and using Lemma \ref{lem:confbdd}. To prove the lower bound, we argue by contradiction, considering sequences $\delta_j\to 0$ and $|\xi_j|\to 1$. In view of the upper bound, which we have already proved, we see as in Step 1 of the proof of Lemma \ref{lem:att} that the infimum in \eqref{eq:distw14alt} for $(\delta,\xi) = (\delta_j,\xi_j)$ is attained at some $\Psi=\Psi_{\xi_j'}$. (In fact, for the proof that follows, it is not necessary to choose $\xi_j'$ as the exact minimizer. It would be sufficient to choose $\xi_j'$ so that it attains the infimum in \eqref{eq:distw14alt} up to an additive error $o(\delta_j^4)$.) It remains to estimate
$$
\left\| (1+\delta(1)_{\Psi_{\xi_j}})_{\Psi_{\xi_j'}} - 1 \right\|_{W^{1,4}}^4 =
\left\| (1)_{\Psi_{\xi_j'}} - 1 + \delta_j (1)_{\Psi_{\xi_j''}} \right\|_{W^{1,4}}^4 \,,
$$
where $\xi_j''\in B_1(0)$ is defined by $\Psi_{\xi_j''} = \Psi_{\xi_j} \circ \Psi_{\xi_j'}$. In fact, we will show that
\begin{equation}
	\label{eq:distw14altlower}
	\left\| (1)_{\Psi_{\xi_j'}} - 1 + \delta_j (1)_{\Psi_{\xi_j''}} \right\|_{W^{1,4}}^4 \geq \delta_j^4 \left\| (1)_{\Psi_{\xi_j''}} \right\|_{W^{1,4}}^4 + o(\delta_j^4) \,.
\end{equation}
Note that, by Lemma \ref{lem:confbdd}, $\| (1)_{\Psi_{\xi_j''}} \|_{W^{1,4}}^4\gtrsim 1$, so \eqref{eq:distw14altlower} contradicts the choice of $(\delta_j,\xi_j)$. Thus, we have reduced the proof of the lower bound in \eqref{eq:distw14alt} to the proof of \eqref{eq:distw14altlower}.

In order to prove \eqref{eq:distw14altlower}, we first show that
\begin{equation}
	\label{eq:distw14altlowerproof}
	|\xi_j'|\lesssim \delta_j \,.
\end{equation}
To this end, we argue similarly to the case of the $W^{1,2}$-norm. Choosing the identity as a competitor and using the triangle inequality, we find
$$
\delta_j\left\|(1)_{\Psi_{\xi_j}}\right\|_{W^{1,4}}\geq \left\|(1)_{\Psi_{\xi_j'}}-1+\delta_j(1)_{\Psi_{\xi_j''}}\right\|_{W^{1,4}}\geq\left\|(1)_{\Psi_{\xi_j'}}-1\right\|_{W^{1,4}} - \delta_j\left\|(1)_{\Psi_{\xi_j''}}\right\|_{W^{1,4}}\,.
$$
From this, we can infer that $\|(1)_{\Psi_{\xi_j'}}-1\|_{W^{1,4}}\lesssim \delta_j$ by Lemma \ref{lem:confbdd}. Meanwhile,  by the same reasoning as in the proof of Lemma \ref{lem:almostortho}, we see that $\| 1- (1)_{\Psi_j} \|_{W^{1,4}}^4 = \operatorname{const} |\xi_j'|^4 + o(|\xi_j'|^4)$. This implies \eqref{eq:distw14altlowerproof}.

Expanding the fourth power in the integrand on the left side of \eqref{eq:distw14altlower} in powers of $\delta_j$ gives $\|(1)_{\Psi_{\xi_j'}}-1\|_{W^{1,4}}^4+\|\delta_j(1)_{\Psi_{\xi_j''}}\|_{W^{1,4}}^4$ plus some mixed terms that are bounded by
\begin{equation}\label{eq:mixed}
	\int_{\mathbb S^d} |\nabla((1)_{\Psi_{\xi_j'}}-1)|^\alpha |\delta_j\nabla(1)_{\Psi_{\xi_j''}}|^{4-\alpha} \,\mathrm d\omega
	\qquad \text{and} \qquad 
	\int_{\mathbb S^d} |(1)_{\Psi_{\xi_j'}}-1|^\alpha |\delta_j(1)_{\Psi_{\xi_j''}}|^{4-\alpha} \,\mathrm d\omega
\end{equation} 
for $\alpha\in \{1,2,3\}$. We claim that these mixed terms are $o(\delta_j^4)$. By dropping the nonnegative term $\|(1)_{\Psi_{\xi_j'}}-1\|_{W^{1,4}}^4$, we will thus arrive at the claimed lower bound \eqref{eq:distw14altlower}.

To prove that the terms in \eqref{eq:mixed} are $o(\delta_j^4)$, we note that, in view of \eqref{eq:psijac},
$$
\left| (1)_{\Psi_{\xi_j'}}(\omega)-1 \right| + \left| \nabla ( (1)_{\Psi_{\xi_j'}}(\omega)-1) \right| \lesssim |\xi_j'| \,,
\qquad\omega\in\Sph^d \,.
$$
This, together with \eqref{eq:distw14altlowerproof}, implies that the terms in \eqref{eq:mixed} are bounded by a constant times
\begin{equation*}
	\delta_j^4 \int_{\mathbb S^d} |\nabla(1)_{\Psi_{\xi_j''}}|^{4-\alpha} \,\mathrm d\omega
	\qquad \text{and} \qquad 
	\delta_j^4 \int_{\mathbb S^d} | (1)_{\Psi_{\xi_j''}}|^{4-\alpha} \,\mathrm d\omega \,.
\end{equation*} 
The assertion that these terms are $o(\delta_j^4)$ follows from the fact that $|\xi_j''|\to 1$ (as a consequence of \eqref{eq:distw14altlowerproof}) together with Lemma \ref{lem:subcritical}.

\medskip

\textit{Step 4.} We now complete the proof of \eqref{eq:sharppow4}. We take an arbitrary sequence $(\delta_j)\in(0,1]$ with $\delta_j\to 0$. Then, according to \eqref{eq:sharppow4norm}, for each fixed $j$, we have
$$
\lim_{|\xi|\to 1} \lambda_{\delta_j,\xi} = (1 + \delta_j^\frac{4d}{d-4})^{-\frac{d-4}{4d}} \,.
$$
Thus, after discarding finitely many $j$ if necessary, we can choose a $\xi_j\in B_1(0)$ such that $|\lambda_{\delta_j,\xi_j} - 1| \leq c \delta_j$ with the constant $c$ from Step 3. We can also ensure that the terms $o_{|\xi|\to 1}(1)$ in \eqref{eq:sharppow4energy} and \eqref{eq:sharppow4norm} with $\delta=\delta_j$ are $\leq \delta_j^{4d/(d-4)}$, and hence we obtain
\begin{equation}\label{eq:sharppow4energy2} 
	E_2[1+\delta_j(1)_{\Psi_{\xi_j}}] \geq (1 + \delta_j^4) \, E_2[1] -\delta_j^{\frac{4d}{d-4}}
\end{equation}
and
\begin{equation} \label{eq:sharppow4norm2}
	\| 1+\delta_j (1)_{\Psi_{\xi_j}}\|_{\frac{4d}{d-4}}^4 \leq |\mathbb S^d|^{\frac{d-4}{
	d}}\left(1+\frac{d-4}{d}\left(1+\frac{1}{|\mathbb S^d|}\right)\delta_j^{\frac{4d}{d-4}} \right).
\end{equation}
Clearly, we may also assume that $|\xi_j|\to 1$.

From the bound $|\lambda_{\delta_j,\xi_j} - 1| \leq c \delta_j$ it follows by Step 3 that $\inf_\Psi \|(u_{\delta_j,\xi_j})_\Psi - 1 \|_{W^{1,4}}^4$ is comparable to $\delta_j^4$. In particular, this infimum tends to zero, as required. Next, it follows from the choice of $\xi_j$ and the bounds \eqref{eq:sharppow4energy2} and \eqref{eq:sharppow4norm2} that
$$
F_2[u_{\delta_j,\xi_j}] - S_d^{(2)} = S_d^{(2)} \delta_j^4 +\mathcal O\left(\delta_j^\frac{4d}{d-4}\right).
$$
This, together with the quartic vanishing of $\inf_\Psi \|(u_{\delta_j,\xi_j})_\Psi - 1 \|_{W^{1,4}}^4$, completes the proof of \eqref{eq:sharppow4}.


\appendix


\section*{Appendix. A strengthening of the result of Figalli and Zhang}\label{app:figallizhang}

After attending a conference presentation of the main results of this paper, Robin Neumayer suggested that a similar two-term stability inequality might hold for the Sobolev inequality in $\dot W^{1,p}(\R^d)$, $p>2$, as well. We are grateful to her for raising this question, and we provide the details of the argument in this appendix.

The sharp Sobolev inequality on $\R^d$ for $1\leq p<d$ reads
$$
\int_{\R^d} |\nabla u|^p\,\mathrm dx \geq S_{d,p} \left( \int_{\R^d} |u|^\frac{dp}{d-p} \,\mathrm dx \right)^\frac{d-p}d
\qquad\text{for all}\ u\in \dot W^{1,p}(\R^d) \,.
$$
Here $\dot W^{1,p}(\R^d)$ denotes the homogeneous Sobolev space, that is, the space of all $u\in L^1_\loc(\R^d)$ satisfying $\nabla u\in L^p(\R^d)$ (where the gradient is taken in the distributional sense) as well as $|\{ |u|>\tau \}|<\infty$ for all $\tau>0$. By definition $S_{d,p}$ denotes the optimal (that is, largest) constant for which this inequality is valid. Let $\mathcal M=\mathcal M_{d,p}$ denote the set of optimizers in the inequality.

\begin{theorem}
	Let $2<p<d$. Then there is a constant $c_{d,p}>0$ such that for all $u\in\dot W^{1,p}(\R^d)$,
	\begin{equation}\label{eq:strengFZ}\tag{A.1}
	 \|\nabla u\|_p^p-S_{p,d}\|u\|_{\frac{dp}{d-p}}^p \geq c_{d,p} \inf_{Q\in\mathcal M} \left( \| \nabla u - \nabla Q \|_p^p + \int_{\R^d} |\nabla Q|^{p-2} |\nabla u-\nabla Q|^2\,\mathrm dx \right).
	\end{equation}
\end{theorem}

\begin{proof}
	Recall that by \cite[Proposition 4.1]{Figalli2015} there is a $\delta_0=\delta_0(d,p)>0$ with the following property: for all $u\in \dot W^{1,p}(\R^d)$ with \begin{equation}\label{eq:cond}\tag{A.2}
			\delta(u)\coloneqq 	\|\nabla u\|_p^p-S_{p,d}\|u\|_{\frac{dp}{d-p}}^p\leq \delta_0\|\nabla u\|_{p}^p
		\end{equation} there is a $Q_0\in \mathcal M$ that attains the infimum in $\inf_{Q\in\mathcal M}
		\int_{\R^d}|\nabla Q|^{p-2}|\nabla u-\nabla Q|^2\,\mathrm dx$.
		(To be precise, the attainability is stated for a different but equivalent distance; see \cite[Remark 2.1]{Figalli2015}. However, the proof of attainability is the same.) 
		
		The infimum on the right side of \eqref{eq:strengFZ} is not greater than $$\|\nabla u-\nabla Q_0\|_p^p+\int_{\R^d}|\nabla Q_0|^{p-2}|\nabla u-\nabla Q_0|^2\,\mathrm dx \,.$$ 
		Combining this with \cite[Proposition 2.4]{Figalli2015}, we obtain for $u\in W^{1,p}(\R^d)$ satisfying \eqref{eq:cond},
		\begin{equation*}
			\inf_{Q\in\mathcal M}\left(\|\nabla u-\nabla Q\|_p^p+\int_{\R^d}|\nabla Q|^{p-2}|\nabla u-\nabla Q|^2\,\mathrm dx\right)\lesssim \delta(u) +\|\nabla u-\nabla Q_0\|_p^p \,.
		\end{equation*} 
		
		It remains to prove that the second term on the right side can be bounded by the first term on the right side. To this end, we use the fact that
		$$
		\|\nabla u-\nabla Q_0\|_p^p
		\lesssim  \inf_{Q\in\mathcal M}\|\nabla u-\nabla Q\|_p^p
		$$
		with an implied constant depending only on $p$ and $d$. This is shown in the proof of \cite[Proposition 4.1(2)]{Figalli2015} (see the first displayed equation in their proof of part (2)). Combining this inequality with the main result from \cite{Figalli2022}, we deduce that indeed $\|\nabla u-\nabla Q_0\|_p^p\lesssim \delta(u)$.	This establishes the inequality in the theorem under the assumption $\delta(u)\leq \delta_0 \|\nabla u\|_{p}^p$.
		
		In case $\delta(u)> \delta_0 \|\nabla u\|_{p}^p$, the claim follows from the straight-forward estimate 
		$$	 
		\inf_{Q\in\mathcal M}\left(\|\nabla u-\nabla Q\|_p^p+\int_{\R^d}|\nabla Q|^{p-2}|\nabla u-\nabla Q|^2\,\mathrm dx\right) \leq \|\nabla u\|_{p}^p \,,
		$$
		obtained by taking $Q=0$. This completes the proof of the theorem.
\end{proof}

\newcommand{\etalchar}[1]{$^{#1}$}
\providecommand{\bysame}{\leavevmode\hbox to3em{\hrulefill}\thinspace}
\providecommand{\MR}{\relax\ifhmode\unskip\space\fi MR }
\providecommand{\MRhref}[2]{%
	\href{http://www.ams.org/mathscinet-getitem?mr=#1}{#2}
}
\providecommand{\href}[2]{#2}



\begin{thebibliography}{CFMP09}
	
	\bibitem[Aub76]{Aubin1976}
	T.~Aubin, \emph{Equations diff\'erentielles non lin\'eaires et probl\`eme
		de {Y}amabe concernant la courbure scalaire}, J. Math. Pures Appl.
	\textbf{55} (1976), 269--296.
	
	\bibitem[BE91]{Bianchi1991}
	G.~Bianchi and H.~Egnell, \emph{A note on the sobolev inequality}, J.
	Funct. Anal. \textbf{100} (1991), no.~1, 18--24.
	
	\bibitem[BDNS23]{Bonforte2023}
	M.~Bonforte, J.~Dolbeault, B.~Nazaret, and N.~Simonov,
	\emph{Constructive stability results in interpolation inequalities and
		explicit improvements of decay rates of fast diffusion equations}, Discrete
	Contin. Dyn. Syst. \textbf{43} (2023), no.~3 \& 4, 1070--1089.
	
	\bibitem[BDNS25]{Bonforte2025}
	M.~Bonforte, J.~Dolbeault, B.~Nazaret, and N.~Simonov, \emph{Stability in gagliardo-nirenberg-sobolev inequalities: Flows,
		regularity and the entropy method}, Memoirs of the AMS, to appear (2025), arXiv:2007.03674.
		
			\bibitem[Bre05]{Brendle2005}
		S.~Brendle, \emph{Convergence of the {Y}amabe flow for arbitrary initial
			energy}, J. Differential Geom. \textbf{69} (2005), no.~2, 217--278.
		
		\bibitem[Bre07]{Brendle2007}
		S.~Brendle, \emph{Convergence of the {Y}amabe flow in dimension 6 and higher},
		Invent. Math. \textbf{170} (2007), no.~3, 541--576.
	
	\bibitem[BL85]{BREZIS198573}
	H.~Brezis and E.~H.~Lieb, \emph{Sobolev inequalities with remainder
		terms}, J. Funct. Anal. \textbf{62} (1985), no.~1, 73--86.
		
			\bibitem[BDS24]{BrDoSi}
		G.~Brigati, J.~Dolbeault, and N.~Simonov, \emph{Logarithmic
			{S}obolev and interpolation inequalities on the sphere: constructive
			stability results}, Ann. Inst. H. Poincar\'e{} C Anal. Non Lin\'eaire
		\textbf{41} (2024), no.~5, 1289--1321.
	
	\bibitem[CCR15]{Carlotto2015}
		A.~Carlotto, O.~Chodosh, and Y.~A.~Rubinstein, \emph{Slowly
		converging yamabe flows}, Geom. Topol. \textbf{19} (2015), no.~3, 1523--1568.
		
	\bibitem[Cas20]{Case2020}
		J.~S.~Case, \emph{The {Frank-Lieb} approach to sharp {S}obolev
			inequalities}, Commun. Contemp. Math. \textbf{23} (2020), no.~03, 2050015.
			
				\bibitem[CW20]{Case2020a}
			J.~S.~Case and Y.~Wang, \emph{Towards a fully nonlinear sharp sobolev
				trace inequality}, J. Math. Study \textbf{53} (2020), no.~4, 402--435.
	
	\bibitem[CFMP09]{Cianchi2009}
	A.~Cianchi, N.~Fusco, F.~Maggi, and A.~Pratelli, \emph{The
		sharp {S}obolev inequality in quantitative form}, J. Eur. Math. Soc.
	\textbf{11} (2009), no.~5, 1105--1139.
	
	\bibitem[CGY02a]{Chang2002}
	S.-Y.~A.~Chang, M.~J.~Gursky, and P.~C.~Yang, \emph{An a priori
		estimate for a fully nonlinear equation on four-manifolds}, J. Anal. Math.
	\textbf{87} (2002), no.~1, 151--186.
	
	\bibitem[CGY02b]{Chang2002a}
	S.-Y.~A.~Chang, M.~J.~Gursky, and P.~C.~Yang, \emph{{An Equation of Monge-Ampere Type in Conformal Geometry, and
			Four-Manifolds of Positive Ricci Curvature}}, Ann. Math. \textbf{155} (2002),
	no.~3, 709--787.
	
		\bibitem[CFW13]{Chen2013}
	S.~Chen, R.~L.~Frank, and T.~Weth, \emph{Remainder {T}erms in the
		{F}ractional {S}obolev {I}nequality}, Indiana Univ. Math. J. \textbf{62}
	(2013), no.~4, 1381--1397.
	
	\bibitem[DEF{\etalchar{+}}23]{Dolbeault2023}
	J.~Dolbeault, M.~J.~Esteban, A.~Figalli, R.~L.~Frank, and M.~Loss, \emph{Sharp stability for {S}obolev and log-{S}obolev inequalities,
		with optimal dimensional dependence}, Preprint (2023), arXiv:2209.08651.
	
	\bibitem[ENS22]{Engelstein2020}
	M.~Engelstein, R.~Neumayer, and L.~Spolaor, \emph{Quantitative stability
		for minimizing {Y}amabe metrics}, Trans. Amer. Math. Soc. Ser. B \textbf{9}
	(2022), 395--414.
	
	\bibitem[FN19]{Figalli2015}
	A.~Figalli and R.~Neumayer, \emph{Gradient stability for the {S}obolev
		inequality: the case $p\geq 2$}, J. Eur. Math. Soc. \textbf{29} (2019),
	no.~2, 319--354.
	
		\bibitem[FZ22]{Figalli2022}
	A.~Figalli and Y.~R.-Y. Zhang, \emph{Sharp gradient stability for the
		{S}obolev inequality}, Duke Math. J. \textbf{171} (2022), no.~12, 2407--2459.
	
	\bibitem[Fra22]{Frank2023a}
	R.~L.~Frank, \emph{Degenerate stability of some {S}obolev inequalities},
	Ann. Inst. H. Poincaré Anal. Non Linéaire \textbf{39} (2022), no.~6,
	1459--1484.
	
	\bibitem[Fra24]{Frank2023}
	R.~L.~Frank, \emph{The sharp {S}obolev inequality and its stability: An
		introduction}, Geometric and Analytic Aspects of Functional Variational
	Principles. Cetraro, Italy 2022 (A.~Cianchi, V.~Maz'ya, and T.~Weth,
	eds.), Springer Cham, 2024, 1--64.
	
		\bibitem[FL12a]{Frank2012a}
	R.~L.~Frank and E.~H.~Lieb, \emph{A new, rearrangement-free proof of
		the sharp hardy-littlewood-sobolev inequality}, Spectral Theory, Function
	Spaces and Inequalities (B.~Malcolm Brown, Jan Lang, and Ian~G. Wood,
	eds.), Springer Basel, 2012, 55--67.
	
	\bibitem[FL12b]{Frank2012b}
	R.~L.~Frank and E.~H.~Lieb, \emph{Sharp constants in several inequalities on the {H}eisenberg
		group}, Ann. of Math. \textbf{176} (2012), no.~1, 349--381.
	
		\bibitem[FP24]{Frank2024}
	R.~L.~Frank and J.~W.~Peteranderl, \emph{Degenerate stability of the
		{C}affarelli-{K}ohn-{N}irenberg inequality along the {F}elli-{S}chneider
		curve}, Calc. Var. Partial Differential Equations \textbf{63} (2024), no.~44.
		
			\bibitem[GW05]{Ge2005}
		Y.~Ge and G.~Wang, \emph{On a fully nonlinear {Y}amabe problem}, Ann.
		Sci. \'Ec. Norm. Sup\'er. \textbf{39} (2005), 569--598.
		
		\bibitem[GW13]{Ge2013}
		Y.~Ge and G.~Wang, \emph{On a conformal quotient equation. {II}},
		Commun. Anal. Geom. \textbf{21} (2013), no.~1, 1--38.
		
		
		\bibitem[GVW03]{Guan2003b}
		P.~Guan, J.~Viaclovsky, and G.~Wang, \emph{{Some Properties of the
				Schouten Tensor and Applications to Conformal Geometry}}, Trans. Amer. Math.
		Soc. \textbf{355} (2003), no.~3, 925--933.
		
		\bibitem[GW04]{Guan2004}
		P.~Guan and G.~Wang, \emph{Geometric inequalities on locally
			conformally flat manifolds}, Duke Math. J. \textbf{124} (2004), no.~1,
		177--212.
	
	\bibitem[GLZ23]{guerra2023sharp}
	A.~Guerra, X.~Lamy, and K.~Zemas, \emph{Sharp quantitative
		stability of the {M}\"obius group among sphere-valued maps in arbitrary
		dimension}, Trans. Amer. Math. Soc., to appear (2023), arXiv:2305.19886.
	
	\bibitem[GV07]{Gursky2007}
	M.~J.~Gursky and J.~A.~Viaclovsky, \emph{{Prescribing Symmetric
			Functions of the Eigenvalues of the Ricci Tensor}}, Ann. Math. \textbf{166}
	(2007), no.~2, 475--531.
	
	\bibitem[K{\"o}n24]{Koenig2024}
	T.~K{\"o}nig, \emph{An exceptional property of the one-dimensional
		bianchi--egnell inequality}, Calc. Var. Partial Differential Equations
	\textbf{63} (2024), no.~123.
	
			\bibitem[LP87]{Lee1987}
	J.~M.~Lee and T.~H.~Parker, \emph{The {Y}amabe problem}, Bull. Amer.
	Math. Soc. (N.S.) \textbf{17} (1987), no.~1, 37--91.
	
	\bibitem[LL03]{Li2003}
	A.~Li and Y.~Y.~Li, \emph{On some conformally invariant fully nonlinear
		equations}, Commun. Pure Appl. Math. \textbf{56} (2003), no.~10, 1416--1464.
		
		\bibitem[Lio85a]{Lions1985}
		P.-L.~Lions, \emph{The {C}oncentration-{C}ompactness {P}rinciple in the
			{C}alculus of {V}ariations. {T}he limit case, {P}art 1.}, Rev. Mat. Iberoam.
		\textbf{1} (1985), no.~1, 145--201.
		
		\bibitem[Lio85b]{Lions1985a}
		P.-L.~Lions, \emph{The {C}oncentration-{C}ompactness {P}rinciple in the {C}alculus
			of {V}ariations. {T}he limit case, {P}art 2.}, Rev. Mat. Iberoam. \textbf{1}
		(1985), no.~2, 45--121.
	
	\bibitem[Neu20]{Neumayer2019}
	R.~Neumayer, \emph{A note on strong-form stability for the {S}obolev
		inequality}, Calc. Var. Partial Differential Equations \textbf{59} (2020),
	no.~25.
	
	\bibitem[Oba71]{Obata1971}
	M.~Obata, \emph{The conjectures on conformal transformations of {R}iemannian
		manifolds}, J. Differential Geom. \textbf{6} (1971), no.~2, 247--258.
	
	\bibitem[Rod66]{Rodemich1966}
	E.~Rodemich, \emph{The {S}obolev inequality with best possible constant},
	Analysis Seminar Caltech (1966).
	
	\bibitem[Sch84]{Schoen1984}
	R.~M.~Schoen, \emph{Conformal deformation of a {R}iemannian metric to
		constant scalar curvature}, J. Differential Geom. \textbf{20} (1984),
	479--495.
	
	\bibitem[STW07]{Sheng2007}
	W.-M.~Sheng, N.~S.~Trudinger, and X.-J.~Wang, \emph{The {Y}amabe problem
		for higher order curvatures}, J. Differential Geom. \textbf{77} (2007),
	no.~3, 515--553.
	
	\bibitem[STW12]{Sheng2012}
	W.-M.~Sheng, N.~S.~Trudinger, and X.-J.~Wang, \emph{The k-{Y}amabe problem}, International Press (2012), 427--457.
	
	\bibitem[SW90]{Stein1990}
	E.~M.~Stein and G.~Weiss, \emph{{Introduction to Fourier Analysis on
			Euclidean Spaces. Sixth Printing}}, Princeton University Press, 1990.
	
	\bibitem[Tal76]{Talenti1976}
	G.~Talenti, \emph{{B}est constant in {S}obolev inequality}, Ann. Mat. Pura
	Appl. \textbf{110} (1976), no.~1, 353--372.
	
	\bibitem[Tru68]{Trudinger1968}
	N.~S.~Trudinger, \emph{Remarks concerning the conformal deformation of
		riemannian structures on compact manifolds}, Ann. Sc. Norm. Super. Pisa -
	Scienze Fisiche e Mat. \textbf{22} (1968), no.~2, 265--274.
	
	\bibitem[Via00]{Viaclovsky2000}
	J.~A.~Viaclovsky, \emph{Conformal geometry, contact geometry, and the
		calculus of variations}, Duke Math. J. \textbf{101} (2000), no.~2, 283--316.
	
	\bibitem[Yam60]{Yamabe1960}
	H.~Yamabe, \emph{On a deformation of {R}iemannian structures on compact
		manifolds}, Osaka Math. J. \textbf{12} (1960), no.~1, 21--37.
	
\end{thebibliography}
\end{document}